\documentclass[12pt,reqno,a4paper]{amsart}
\usepackage{amssymb}

\usepackage{graphicx}
\usepackage{diagrams}
\usepackage{comment}

\usepackage{hyperref}

\usepackage{tikz}
\usetikzlibrary{decorations.markings}

\let\nbd=\nobreakdash
\newcommand\inv{^{-1}}

\makeindex
\newtheorem{theorem}{Theorem}[section]
\newtheorem{proposition}[theorem]{Proposition}
\newtheorem{corollary}[theorem]{Corollary}
\newtheorem{lemma}[theorem]{Lemma}
\theoremstyle{definition}
\newtheorem{definition}[theorem]{Definition}

\newcommand{\id}{\mathrm{id}}

\newcommand{\ie}{{\it i.e.}}

\newcommand{\viz}{{\it viz.}}
\newcommand{\powers}[1]{[\kern -1pt [{#1}] \kern -1pt ]} 
\newcommand{\nov}[1]{(\kern -1.7pt ( {#1})\kern-1.7pt )} 
\newcommand{\tensor}{\mathop{\otimes}}
\newcommand{\bZ}{\mathbb{Z}}
\newcommand{\bR}{\mathbb{R}}
\newcommand{\bN}{\mathbb{N}}

\newcommand{\fS}{\mathfrak{S}}

\newcommand{\cN}{\mathcal{N}}
\newcommand{\cY}{\mathcal{Y}}

\newcommand{\rk}{\mathrm{rk}}

\newcommand{\iso}{\cong}
\newcommand{\nix}{\,\text{-} \,}

\newcommand{\tot}{\mathrm{Tot}}

\renewcommand{\phi}{\varphi}

\newcommand{\Ch}{\textsf{Ch}\,}

\newcommand{\tora}[1]{{\mathfrak{T}({#1})}}
\newcommand{\torb}[3]{\tora{{#1},{#2};{#3}}}

\newcommand{\R}{R_*}

\newcommand{\lr}[1]{A \langle #1 \rangle}

\newcommand{\Mod}{\text{-}\mathrm{Mod}}

\numberwithin{equation}{theorem}
\numberwithin{section}{part}   

\begin{document}

\author{Thomas H\"uttemann} \author{Luke Steers}

\email{t.huettemann@qub.ac.uk}

\address{Pure Mathematics Research Centre \\
  School of Mathematics and Physics       \\
  Queen's University Belfast              \\
  Belfast BT7 1NN                         \\
  Northern Ireland, UK}

\title[Finite domination over strongly $\bZ^2$-graded rings]%
{Finite domination and \textsc{Novikov} homology over strongly $\bZ^2$-graded
  rings}

\dedicatory{Dedicated to the memory of Andrew Ranicki}

\date{\today}

\makeatletter
\@namedef{subjclassname@2020}{{\upshape 2020} Mathematics Subject Classification}
\makeatother


\begin{abstract}
  Let \(R\) be a strongly \(\bZ^2\)-graded ring, and let \(C\) be a
  bounded chain complex of finitely generated free \(R\)-modules. The
  complex \(C\) is \(R_{(0,0)}\)-finitely dominated, or of type~\(FP\)
  over~\(R_{(0,0)}\), if it is chain homotopy equivalent to a bounded
  complex of finitely generated projective \(R_{(0,0)}\)-modules. We
  show that this happens if and only if \(C\) becomes acyclic after
  taking tensor product with a certain eight rings of formal power
  series, the graded analogues of classical \textsc{Novikov}
  rings. This extends results of \textsc{Ranicki}, \textsc{Quinn} and
  the first author on \textsc{Laurent} polynomial rings in one and
  two indeterminates.
\end{abstract}

\maketitle

\tableofcontents

\part{Finite Domination over strongly $\bZ^2$-graded rings}

\section{Introduction}

Let $L$ be a unital ring, and let $K$ be a subring of~$L$. A bounded
chain complex~$C$ of (right) $L$-modules is {\it $K$-finitely
  dominated\/} if $C$, considered as a complex of $K$-modules, is a
retract up to homotopy of a bounded complex of finitely generated free
$K$-modules; this happens if and only if $C$ is homotopy equivalent,
as a $K$-module complex, to a bounded complex of finitely generated
projective $K$-modules \cite[Proposition~3.2.~(ii)]{RATFO}. The
following result of \textsc{Ranicki} gives a complete homological
characterisation of finite domination in an important special case:

\begin{theorem}[\textsc{Ranicki \cite[Theorem~2]{RFDNR}}]
  \label{thm:orinigal}
  Let $K$ be a unital ring, and let $K[t,t\inv ]$ denote the
  \textsc{Laurent} polynomial ring in the indeterminate~$t$. Let $C$
  be a bounded chain complex of finitely generated free
  $K[t,t\inv ]$-modules. The complex~$C$ is $K$-finitely
  dominated if and only if both
  \begin{equation*}
    C \tensor_{K[t,t\inv ]} K\nov{t\inv } \qquad \text{and} \qquad C
    \tensor_{K[t,t\inv ]} K\nov{t}
  \end{equation*}
  have vanishing homology in all degrees. Here we use the notation
  $K\nov{t} = K\powers{t}[t\inv ]$ for the ring of formal
  \textsc{Laurent} series in~$t$, and similarly
  $K\nov{t\inv } = K\powers{t\inv }[t]$ stands for the ring of
  formal \textsc{Laurent} series in~$t\inv $.
\end{theorem}

The cited paper \cite{RFDNR} also contains a discussion of the
relevance of finite domination in topology. --- The
rings~$K\nov{t}$ and~$K\nov{t\inv }$ are known as
\textsc{\textsc{Novikov}} rings. The theorem can be formulated more succinctly:
{\it The chain complex~$C$ is $K$\nobreakdash-finitely dominated
  if and only if it has trivial \textsc{\textsc{Novikov}} homology}.

This result was extended by \textsc{Quinn} and the first author to
\textsc{Laurent} polynomial rings in two variables; contrary to
appearance, this is not a straight-forward modification of the
original result, introducing additional levels of complication in
homological algebra.

\begin{theorem}[\textsc{H\"uttemann and Quinn \cite[Theorem~I.1.2]{TWOV}}]
  \label{thm:old_two}
  Let $C$ be a bounded chain complex of finitely generated free
  $L$\nbd-modules, where $L = K[x,x\inv, y, y\inv]$ is a
  \textsc{Laurent} polynomial ring in two variable over the unital
  ring~$K$. The following two statements are equivalent:
  \begin{enumerate}
  \item \label{item:findom} The complex~$C$ is $K$\nbd-finitely
    dominated, \ie, $C$~is homotopy equivalent, as an $K$\nbd-module
    chain complex, to a bounded chain complex of finitely generated
    projective $K$\nbd-modules.
  \item \label{item:acyclic} The eight chain complexes listed below
    are acyclic (all tensor products are taken over~$L$):
    \begin{subequations}
      \begin{gather}\left.
          \begin{aligned}
            \label{eq:cond_edge}
 & C \tensor K[x,\, x\inv]\nov{y}        &  & C \tensor K[x,\,
            x\inv]\nov{y\inv}                    \\
 & C \tensor K[y,\, y\inv]\nov{x}        &  & C \tensor K[y,\,
            y\inv]\nov{x\inv}
          \end{aligned}\quad \right\}
                                                 \\
        \noalign{\smallskip}\left.
          \begin{aligned}
            \label{eq:cond_vertex}
 & C \tensor K\nov{x,\, y} \hskip 2.1 em &  & C \tensor
            K\nov{x\inv,\, y\inv} \hskip 1.15 em \\
 & C \tensor K\nov{x,\, y\inv}           &  & C \tensor K\nov{x\inv,\,
              y}
          \end{aligned}\quad \right\}
      \end{gather}
    \end{subequations}
    Here $K\nov{x,y} = K\powers{x,y}[1/xy]$ is a localisation of the
    ring of formal power series in~$x$ and~$y$, and the other rings
    are defined analogously.
  \end{enumerate}
\end{theorem}

The authors of the present paper generalised
Theorem~\ref{thm:orinigal} in an entirely different direction,
exhibiting the graded structure of \textsc{Laurent} polynomial rings
as the crucial property for setting up the theory.

\begin{theorem}[\textsc{H\"uttemann and Steers \cite[Theorem~1.3]{ZGR}}]
  \label{thm:new_one}
  Let $\R[t,t\inv ] = \bigoplus_{k \in \bZ} R_{k}$ be a strongly
  $\bZ$-graded ring, and let $C$ be a bounded chain complex of
  finitely generated free $\R[t,t\inv ]$-modules. The complex~$C$ is
  $R_{0}$-finitely dominated if and only if both
  \begin{equation*}
    \label{eq:cond}
    C \tensor_{\R[t,t\inv ]} \R\nov{t\inv } \qquad \text{and} \qquad C
    \tensor_{\R[t,t\inv ]} \R\nov{t}
  \end{equation*}
  have vanishing homology in all degrees. Here the rings
  \begin{equation*}
    \R\nov{t\inv }= \bigcup_{n \geq 0} \prod_{k \leq n} R_k \quad
    \text{and} \quad 
    \R\nov{t} = \bigcup_{n \geq 0} \prod_{k \geq -n} R_k
  \end{equation*}
  are used as formal analogues of the usual \textsc{\textsc{Novikov}}
  rings.
\end{theorem}

The notion of a {\it strongly graded ring\/} will be discussed in
detail below. In the first instance, $R = \bigoplus_{k \in \bZ} R_{k}$
is a $\bZ$-graded ring. The (usual) \textrm{Laurent} polynomial ring
$R[t,t\inv]$ has a $\bZ$-graded subring $\R[t,t\inv]$ with $k$th
component the set of monomials $r_{k} t^{k}$ with $r_{k} \in R_{k}$.
In fact, we may identify $\R [t,t\inv]$ with $R$ itself. In a similar
spirit, the \textsc{Novikov} rings $R \nov{t\inv}$ and $R \nov{t}$
have subrings $\R \nov{t\inv}$ and $\R \nov{t}$ determined by the
condition that the coefficient of $t^{k}$ be an element of~$R_{k}$,
for any $k \in \bZ$. It may be worth pointing out that these subrings
do not contain the indeterminate~$t$, which should be considered a
purely notational device.

\section{The main theorem}

In the present paper we take the step to strongly $\bZ^2$-graded
rings, combining ideas from both of the aforementioned publications
\cite{TWOV} and~\cite{ZGR}. Roughly speaking, a bounded chain
complex~$C$ of finitely generated free modules over a strongly
$\bZ^2$\nbd-graded ring is finitely dominated over the degree-$0$
subring if and only if certain eight complexes induced from~$C$ are
acyclic. Indeed, from a $\bZ^{2}$-graded ring
$R = \bigoplus_{\sigma \in \bZ^{2}} R_{\sigma}$ we construct the
following eight \textsc{\textsc{Novikov}}-type rings:

\begin{equation}
  \label{eq:Novikov_type}
  \left.
  \begin{aligned}
    \R[x,x\inv] \nov{y}     & = \bigcup_{n \geq 0} \prod_{y \geq -n}
                          \bigoplus_{x \in \bZ} R_{(x,y)}                         \\
    \R[x,x\inv] \nov{y\inv} & = \bigcup_{n \geq 0} \prod_{y \geq -n}
                              \bigoplus_{x \in \bZ} R_{(x,-y)}                    \\
    \R[y,y\inv] \nov{x}     & = \bigcup_{n \geq 0} \prod_{x \geq -n}
                              \bigoplus_{y \in \bZ} R_{(x,y)}                     \\
    \R[y,y\inv] \nov{x\inv} & = \bigcup_{n \geq 0} \prod_{x \geq -n}
                              \bigoplus_{y \in \bZ} R_{(-x,y)}                    \\
    \R \nov{x,y}            & = \bigcup_{n \geq 0} \prod_{x,y \geq -n} R_{(x,y)}
                                                                                  \\
    \R \nov {x, y\inv}      & = \bigcup_{n \geq 0} \prod_{x,y \geq -n} R_{(x,-y)} \\
    \R \nov{x\inv,y\inv}    & = \bigcup_{n \geq 0} \prod_{x,y \geq -n}
                           R_{(-x,-y)}                                            \\
    \R \nov {x\inv, y}      & = \bigcup_{n \geq 0} \prod_{x,y \geq -n} R_{(-x,y)} \\
  \end{aligned}
  \qquad \right\}
\end{equation}

\noindent Similar to notation used earlier, the symbols~$x$ and~$y$ do
{\it not\/} stand for actual indeterminates; they are purely
notational devices, emphasising a formal similarity with
\textsc{Laurent} polynomial rings and their associated
\textsc{\textsc{Novikov}} rings. The ring \(\R[x,x\inv]\nov{y}\) is
the subring of \(R[x,x\inv]\nov{y}\) with elements
\(\sum_{(a,b) \in \bZ^2} r_{a,b} x^a y^b\) such that
\(r_{a,b} \in R_{(a,b)}\), and similar for the other seven cases. ---
With this notational camouflage we obtain a perfect analogue of
Theorem~\ref{thm:old_two}:
 
\begin{theorem}
  \label{thm:main}
  Let $R=\bigoplus_{k\in\bZ^2}R_k$ be a strongly $\bZ^2$-graded ring,
  and let $C$ be a bounded chain complex of finitely generated free
  $R$-modules.  The complex $C$ is $R_{(0,0)}$-finitely dominated if
  and only if all of the eight complexes
\begin{subequations}
      \begin{gather}\left.
          \begin{aligned}
            \label{eq:cond_edge_new}
               & C \tensor_{R} R_{*}[x,\, x\inv]\nov{y}        & \quad & C \tensor_{R}
            R_{*}[x,\,
            x\inv]\nov{y\inv}                         \\[1.5ex]
               & C \tensor_{R} R_{*}[y,\, y\inv]\nov{x}        &       & C \tensor_{R}
            R_{*}[y,\, y\inv]\nov{x\inv}
          \end{aligned}\quad \right\}
                                                      \\
        \intertext{and}
        \left.
          \begin{aligned}
            \label{eq:cond_vertex_new}
            \, & C \tensor_{R} R_{*}\nov{x,\, y} \hskip 2.0 em & \quad & C \tensor_{R}
            R_{*}\nov{x\inv,\, y\inv} \hskip 1.15 em  \\[1.5ex]
               & C \tensor_{R} R_{*}\nov{x,\, y\inv}           &       & C \tensor_{R}
            R_{*}\nov{x\inv,\, y} \ ,
          \end{aligned}\quad \right\}
      \end{gather}
    \end{subequations}
    have vanishing homology in all degrees.
\end{theorem}

To illustrate the extent of the generalisation we describe a strongly
$\bZ^2$-graded ring~\(\hat{K}\) that is not a crossed product (and in
particular not a \textsc{Laurent} polynomial ring), and a complex of
finitely generated free \(K\)-modules that is
\(\hat{K}_{(0,0)}\)-finitely dominated by the above criterion.

Let \(\overline K = K[a, b, c, d] / ab+cd-1\); this is a strongly
$\bZ$-graded ring if we let $a, c$ have degree~$1$ and $b, d$ have
degree~$-1$. It is not a \textsc{Laurent} polynomial ring, in fact not
a crossed product, since the only units in $\overline{K}$ are the
elements of $ K^\times$ in degree~\(0\), as can be shown using ideas
from \textsc{Gr\"obner} basis theory. In particular, there is no unit
in~\({\overline K}_{1}\).

The \(\bZ^2\)-graded ring
$\hat K = \overline K \tensor_{{\overline K}_0} \overline K$, where
$\hat K_{(k,\ell)} = {\overline K}_k \tensor_{{\overline K}_0}
{\overline K}_\ell$, can be seen to be strongly graded as the
necessary partitions of unity are present as in point $(3)$ of
Proposition~\ref{prop:characterisation_strongly_graded}. It is not a
crossed product since, for example, there is no unit in
\(\hat{K}_{(1,0)} = {\overline K}_{1}\).

We will define a four-fold chain complex~\(V\) concentrated in degrees
\((\epsilon_1, \epsilon_2, \epsilon_3, \epsilon_4) \in \bZ^4\) with
\(\epsilon_j = 0,1\). The entries are \(\hat{K}\), a free
\(\hat{K}\)-module of rank~1; the non-trivial differentials are
multiplication by
\begin{align*}
  & 1 - a \tensor 1   & \text{ in 1-direction,} \\
  & 1 - 1 \tensor c   & \text{ in 2-direction,} \\
  & 1 - d \tensor 1   & \text{ in 3-direction,} \\
  & 1 - 1 \tensor bcd & \text{ in 4-direction.}
\end{align*}
That is, the map from position
\((1,\epsilon_2, \epsilon_3, \epsilon_4)\) to position
\((0,\epsilon_2, \epsilon_3, \epsilon_4)\) is multiplication
by \(1 - a \tensor 1\), the map from position
\((\epsilon_1, 1, \epsilon_3, \epsilon_4)\) to position
\((\epsilon_1, 0, \epsilon_3, \epsilon_4)\) is multiplication
by \(1 - 1 \tensor c\), and so on.

The totalisation~\(C\) of~\(V\) can be obtained by taking ``iterated
mapping cones'', and up to isomorphism it does not matter in which
order we choose the different directions. Now the map
\(1 - a \tensor 1\) is an isomorphism with inverse
\begin{equation*}
  (1 - a \tensor 1)\inv = 1 + a \tensor 1 + a^2 \tensor 1 + a^3
  \tensor 1 + \ldots
\end{equation*}
over the rings \(\hat{K}_* [y,y\inv]\nov{x}\) and
\(\hat{K}_* \nov{x, y^{\pm 1}}\), by the usual telescoping sum
argument familiar from the geometric series. Note that the series
defines an element in each of these rings since
\(a^k \tensor 1 = (a \tensor 1)^k\) has degree~\((k,0)\). Thus after
tensoring \(V\) with one of these rings the mapping cones in
1\nbd-direction will be acyclic, hence the tensor product of~\(C\)
with one of these rings will be acyclic.

Similarly, the map \(1 - 1 \tensor bcd \) is an isomorphism over
\(\hat{K}_* [x, x\inv]\nov{y\inv}\) and
\(\hat{K}_* \nov{x^{\pm 1}, y\inv}\), with inverse
\begin{equation*}
  (1 - 1 \tensor bcd)\inv = 1 + 1 \tensor bcd + 1 \tensor (bcd)^2 + 1
  \tensor (bcd)^3 + \ldots \ ;
\end{equation*}
this is an element of the rings in question since the degree of
\(1 \tensor (bcd)^k\) is~\((0,-k)\). Consequently, after tensoring
\(V\) with one of these rings the mapping cones in 4\nbd-direction
will be acyclic, hence the tensor product of~\(C\) with one of these
rings will be acyclic. The cases of \(1 - 1 \tensor c\) and
\(1 - d \tensor 1\) are similar, involving the rings
\(\hat{K}_* [x, x\inv]\nov{y}\) and \(\hat{K}_* \nov{x^{\pm 1}, y}\)
in the former case, the rings \(\hat{K}_* [y, y\inv]\nov{x\inv}\)
and \(\hat{K}_* \nov{x\inv, y^{\pm 1}}\) in the latter.

\subsection*{Structure of the paper}

The paper can be seen as an amalgamation of the two publications
\cite{TWOV} and~\cite{ZGR}, combining the graded viewpoint of the
latter with the elaborate homological algebra of the former. Replacing
central indeterminates by inherently non-commutative structures is a
non-trivial task, and it is rather surprising that with the right
set-up the overall pattern of proof remains virtually unchanged.

We start with recalling basic concepts from the theory of strongly
graded rings. The ``if'' implication of the main theorem is verified
in Part~\ref{part:contr-textscn-homol}, building on \cite{ZGR} and
\cite{TWOV}; the main point is to relate the given chain complex with
a complex of diagrams which are the analogues of well-known line
bundles on the scheme \(\mathbb{P}^1 \times
\mathbb{P}^1\). Part~\ref{part:finite-domin-impl} focuses on the
``only if'' implication. The main technical device is the algebraic
torus (a ``two-dimensional'' version of the algebraic mapping torus),
and the \textsc{Mather} trick which allows to replace the
complex~\(C\) with an algebraic torus of a complex~\(D\) consisting of
finitely generated projective \(R_{(0,0)}\)-modules.

\section{Graded rings and \textsc{Novikov} rings}

\begin{definition}
  A \emph{$\bZ^2$\nbd-graded ring} is a (unital) ring~$L$ equipped
  with a direct sum decomposition into additive subgroups
  $L=\bigoplus_{k\in{\bZ}^2} L_k$ such that
  $L_kL_\ell\subseteq L_{k+\ell}$ for all $ k,\ell\in{\bZ}^2$, where
  $L_kL_\ell$ consists of the finite sums of ring products $xy$ with
  $x \in L_k$ and $y \in L_\ell$.  The summands~$L_{k}$ are called the
  {\it (homogeneous) components\/} of~$L$; elements of~$L_{k}$ are
  called {\it homogeneous of degree~$k$}.
\end{definition}

Relevant examples are the polynomial ring $K[x,y]$ and the
\textsc{Laurent} polynomial ring $R = R_{0}[x,x\inv,y,y\inv]$ in two
variables, equipped with the usual $\bZ^{2}$-grading by the exponents
of~$x$ and~$y$, respectively; here $K$ may be any unital ring.

\medbreak\goodbreak

From a unital ring~$R$, for the moment not equipped with a grading, we
can construct various ``\textsc{\textsc{Novikov}} rings'' of formal
\textsc{Laurent} series as follows:
\begin{align*}
  R[x,x\inv] \nov{y}     & = R[x,x\inv] \powers{y} [y\inv]       \\
  R[y,y\inv] \nov{x}     & = R[y,y\inv] \powers{x} [x\inv]       \\
  R[x,x\inv] \nov{y\inv} & = R[x,x\inv] \powers{y\inv} [y]       \\
  R[y,y\inv] \nov{x\inv} & = R[y,y\inv] \powers{x\inv} [x]       \\
  R \nov{x,y}            & = R \powers{x,y} [(xy)\inv]           \\
  R \nov{x\inv,y\inv}    & = R \powers{x\inv, y\inv} [xy]        \\
  R \nov {x, y\inv}      & = R \powers{x, y\inv} [(x y\inv)\inv] \\
  R \nov {x\inv, y}      & = R \powers{x\inv, y} [(x\inv y)\inv]
\end{align*}
Here $x$ and~$y$ are central indeterminates which commute with each
other, and with the elements of~$R$. The slightly informal notation
$S \powers{x} [x\inv]$ is meant to denote the ring of formal power
series in the indeterminate~$x$ with coefficients in~\(S\), localised
at the multiplicative system $\{x^k \,|\, k \geq 0\}$.

Returning to the case of a $\bZ^{2}$-graded ring~$R$, we can introduce
subrings of rings of formal power or \textsc{Laurent} series, denoted
by adding the decoration ``$*$'' as a subscript to the ring~$R$,
consisting of those elements such that the coefficient of
$x^{s} y^{t}$ is an element of~$R_{(s,t)}$. The resulting unital
subrings will not contain the elements $x$ and~$y$ as, for example,
$x = 1 \cdot x$ has coefficient $1 \in R_{0}$, whereas the
monomial~$x$ has degree $(1,0)$. At this point we may treat the
``variables'' as a purely notational device keeping track of degrees
of various elements. As an illustration we put forward the ring
$\R[x,x\inv,y,y\inv]$ which has elements the finite sums of
terms $r_{s,t} x^{s} y^{t}$, with $s,t \in \bZ$ and
$r_{s,t} \in R_{(s,t)}$. We may in fact identify this ring with the
(graded) ring~$R$ itself. --- When applied to the \textsc{\textsc{Novikov}}
rings listed above, we arrive at the identifications listed
in~\eqref{eq:Novikov_type}.

\section{Partitions of unity and strongly graded rings}

\begin{definition}
  Let $S = \bigoplus_{\sigma \in \bZ^{2}} S_{\sigma}$ be a
  $\bZ^2$-graded unital ring. For $\rho\in\bZ^2$, an finite sum
  expression of the form $1=\sum_{j} {u_j}{v_j}$ with
  $u_j\in S_{\rho}$ and $v_j\in S_{-\rho}$ is called a \emph{partition
    of unity of type $(\rho,-\rho)$}.
\end{definition}

By a straightforward computation one can show:

\begin{lemma}
  \label{lem:new_pou_from_old}
  For $\rho, \rho' \in \bZ^2$ let $1=\sum_{j} u_jv_j$ and
  $1=\sum_{k} u'_kv'_k$ be partitions of unity of type~$\rho$
  and~$\rho'$, respectively. Then
  \begin{equation*}
    1=\sum_{k} \sum_{j} (u_j u'_j) (v'_j v_j)
  \end{equation*}
  is a partition of unity of type $\rho + \rho'$. \qed
\end{lemma}

If $S_{\rho} S_{-\rho} = S_{(0,0)}$ then, since $1 \in S_{(0,0)}$, there
exists a partition of unity of type~$(\rho,-\rho)$. Conversely, if a
partition of unity of type~$(\rho,-\rho)$ exists then
$1 \in S_{\rho} S_{-\rho}$ and thus
$S_{(0,0)} \subseteq S_{\rho} S_{-\rho}$ so that
$S_{(0,0)} = S_{\rho} S_{-\rho}$; this implies moreover
\begin{equation*}
  S_{\sigma} = S_{\sigma} S_{(0,0)} = S_{\sigma} S_{\rho} S_{-\rho}
\subseteq S_{\sigma+\rho} S_{-\rho} \subseteq S_{\sigma}
\end{equation*}
so that $S_{\sigma} = S_{\sigma+\rho} S_{-\rho}$ for any $\sigma \in
\bZ^{2}$, and similarly $S_{\sigma} = S_{\rho} S_{\sigma - \rho}$.

\begin{proposition}
  \label{prop:pi_mu}
  Let $R$ be a $\bZ^{2}$\nbd-graded ring.  Let
  $\lambda, \rho \in \bZ^{2}$ be such that there exists a partition of
  unity $1 = \sum_{j} \alpha_{j} \beta_{j}$ of
  type~$(\lambda, -\lambda)$. The multiplication map
  \begin{equation*}
    \pi_{\lambda,\rho} \colon R_{\lambda} \tensor_{R_{0}} R_{\rho}
    \rTo R_{\lambda + \rho} \ , \quad x \tensor y \mapsto xy
  \end{equation*}
  is an isomorphism of $R_{(0,0)}$\nbd-bimodules; its inverse can be
  written as
  \begin{equation*}
    \mu_{\lambda, \rho} = \pi_{\lambda,\rho}\inv \colon
    R_{\lambda+\rho} \rTo R_{\lambda} \tensor_{R_{(0,0)}} R_{\rho} \ ,
    \quad z \mapsto \sum_{j} \alpha_{j} \tensor \beta_{j} z \ .
  \end{equation*}
  The map $\mu_{\lambda, \rho}$ does not depend on the choice of
  partition of unity.
\end{proposition}

\begin{proof}
  The map $\pi_{\lambda,\rho}$ is an $R_{(0,0)}$\nbd-balanced, thus
  well defined, \(R_{(0,0)}\)\nbd-bi\-module homomorphism. Hence it is
  enough to show that $\mu_{\lambda,\rho}$, considered as a
  homomorphism of right $R_{(0,0)}$\nbd-modules, is its inverse. For
  $z \in R_{\lambda+\rho}$ we calculate
  \begin{equation*}
    \pi_{\lambda,\rho} \circ \mu_{\lambda,\rho} (z) = \sum_{j}
    \alpha_{j} \beta_{j} z = z
  \end{equation*}
  so that $\pi_{\lambda,\rho} \circ \mu_{\lambda,\rho}$ is the
  identity map. Similarly, for $x \in R_{\lambda}$ and $y \in
  R_{\rho}$ we have
  \begin{equation*}
    \mu_{\lambda,\rho} \circ \pi_{\lambda,\rho} (x \tensor y) =
    \sum_{j} \alpha_{j} \tensor (\beta_{j} x) y = \sum_{j} \alpha_{j}
    \beta_{j} x \tensor y = x \tensor y
  \end{equation*}
  (since $\beta_{j} x \in R_{(0,0)}$) so that
  $\mu_{\lambda,\rho} \circ \pi_{\lambda,\rho}$ is the identity map.
\end{proof}

\begin{definition}[{\textsc{Dade} \cite[\S1]{GRD}}]
  A $\bZ^2$\nbd-graded ring
  $L = \bigoplus_{\rho \in \bZ^{2}} L_{\rho}$ is called a {\it
    strongly ${\bZ}^2$-graded ring} if
  $L_{\kappa} L_{\lambda} = L_{\kappa+\lambda}$ for all
  $\kappa, \lambda \in \bZ^{2}$.
\end{definition}

The establishing example is the Laurent polynomial ring in two
variables, $R = R_{0}[x,x\inv,y,y\inv]$. This is a strongly
${\bZ}^2$-graded ring when $x$ and~$y$ are given degrees~$(1,0)$ and
$(0,1)$, respectively.

\medbreak

Using \cite[Proposition~1.6]{GRD} and
Lemma~\ref{lem:new_pou_from_old}, one obtains the following
characterisation of strongly $\bZ^{2}$\nbd-graded rings:

\begin{proposition}
  \label{prop:characterisation_strongly_graded}%
  Let $R$ be a $\bZ^2$-graded ring. The following statements are
  equivalent:
  \begin{enumerate}
  \item \label{item:1} The ring $R$ is strongly graded.
  \item \label{item:2} For every $\rho\in{\bZ}^2$ there is at least
    one partition of unity of type $(\rho,-\rho)$.
  \item \label{item:3} There is at least one partition of unity of
    each of the types
    \begin{itemize}
    \item $\big((1,0),(-1,0)\big)$ and $\big((-1,0),(1,0)\big)$,
    \item $\big((0,1),(0,-1)\big)$ and $\big((0,-1),(0,1)\big)$.\qed
    \end{itemize}
  \item \label{item:4} For every $\rho \in \bZ^{2}$ we have
    $L_{\rho} L_{-\rho} = L_{(0,0)}$.
  \end{enumerate}
\end{proposition}

\begin{corollary}
  \label{cor:components_fgp}
  If $L = \bigoplus_{\kappa \in \bZ^2} L_{\kappa}$ is strongly graded,
  then all $L_{\kappa}$ are invertible $L_{(0,0)}$\nbd-bimodules. In
  particular, the $L_{\kappa}$ are finitely generated projective as
  right and as left $L_{(0,0)}$\nbd-modules. \qed
\end{corollary}

\part{Contractibility of \textsc{Novikov} homology implies finite domination}
\label{part:contr-textscn-homol}

\section{Rings and modules associated to faces of a square}
\label{sec:rings_faces}

We denote by the symbol $\fS$ the set of non-empty faces of the
polytope $S = [-1,1] \times [-1,1] \subset \bR^{2}$, partially ordered
by inclusion. In Fig.~\ref{fig:faces_of_S} we show our chosen labelling and
orientation of the faces (the orientations will be used later to define
incidence numbers).
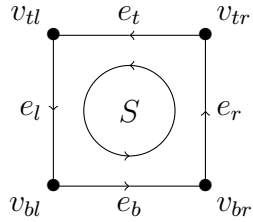
\begin{figure}[ht]
  \centering
  \begin{tikzpicture}
    \draw[->] (0,0) node {\(\bullet\)}
    node[below left] {\(v_{bl}\)}
    -- (1,0) node[below] {\(e_b\)};
    \draw (1,0) -- (2,0); \draw[->] (2,0) node {\(\bullet\)}
    node[below right] {\(v_{br}\)}
    -- (2,1) node[right] {\(e_r\)};
    \draw (2,1) -- (2,2); \draw[->] (2,2) node {\(\bullet\)}
    node[above right] {\(v_{tr}\)}
    -- (1,2) node[above] {\(e_t\)};
    \draw (1,2) -- (0,2); \draw[->] (0,2) node {\(\bullet\)}
    node[above left] {\(v_{tl}\)}
    -- (0,1) node[left] {\(e_l\)};
    \draw (0,1) -- (0,0); \draw[decoration={markings, mark=at position
      0.25 with \arrow{>}, mark=at position 0.75 with {\arrow{>}}},
    postaction={decorate}] (1,1) circle (0.6) node {$S$};
  \end{tikzpicture}
  \caption{Labelling and orientation of the faces of the square~$S$}
  \label{fig:faces_of_S}
\end{figure}

With each face $F \in \fS$ we associate its \textit{barrier cone}
\begin{equation*}
  T_{F} = \mathrm{cone}\, \{x-y \,|\, x \in S,\ y \in F\} \ \subseteq
  \bR^{2} \ ,
\end{equation*}
where $\mathrm{cone}\,(X)$ denotes the set of finite linear
combinations of elements of~$X$ with non-negative real
coefficients. For example, \(T_{v_{bl}}\) is the first quadrant and
\(T_{e_r}\) is the left half-plane. --- Given a $\bZ^{2}$\nbd-graded
ring $R = \bigoplus_{\rho \in \bZ^{2}} R_{\rho}$ we let
$A_F = \R[T_F]$ denote the subring of $R = \R[x, x\inv, y, y\inv]$
consisting of elements with support in~$T_F$; explicitly,
\begin{equation*}
  A_F = \R[T_F] = \bigoplus_{\rho \in T_F \cap \bZ^{2}} R_{\rho} \ .
\end{equation*}
Thus, for example,
\begin{center}
  \begin{tabular}[h]{lll}
    $A_{v_{br}}$ & $= \R[T_{v_{br}}]$ & $= \R[x\inv, y] \ ,$  \\
    $A_{e_{l}}$  & $= \R[T_{e_{l}}]$  & $= \R[x,y,y\inv] \ ,$ \\
    $A_{S}$      & $= \R[T_{S}]$ & $= \R[x, x\inv, y, y\inv] = R \ . $
  \end{tabular}
\end{center}
For $k \in \bZ$ we introduce the $\R[T_{F}]$-bimodules
\begin{equation*}
  \R[kF + T_{F}] = \bigoplus_{\sigma \in (kF + T_{F}) \cap \bZ^{2}} R_{\sigma} \ ,
\end{equation*}
where
$kF + T_{F} = \{ kx + y \,|\, x \in F \text{ and } y \in T_{F}
\}$. In different notation,
\begin{center}
  \begin{tabular}[h]{lll}
    $\R[k v_{br} + T_{v_{br}}]$  & $= x^{-k} y^{-k} \R[x\inv, y] $ & $= x^{-k}y^{-k}
                                                                    A_{v_{br}}
                                                                    \ ,$ \\
    $\R[k e_{l} + T_{e_{l}}]$    & $= x^{-k} \R[x,y,y\inv]$        & $= x^{-k}
                                                            A_{e_{l}}
                                                            \ ,$         \\
    $\R[k S + T_{S}]$            & $= \R[x, x\inv, y, y\inv]$      & $= R \ .$
  \end{tabular}
\end{center}
As a final bit of notation for now, we let $p_{F} \in F \cap \bZ^{2}$
denote the barycentre of~$F$. For example, $p_{e_{t}} = (0,1)$ and
$p_{S} = (0,0)$. Then $kF+T_{F} = kp_{F} + T_{F}$, and every element~$\sigma$
of $kF+T_{F}$ can in fact be written in the form $kp_{F} + \tau$ for some
$\tau \in T_{F}$. Also, given faces $F \subseteq G$ the barrier cones
satisfy $T_{G} = T_{F} + \bR \cdot (p_{G} - p_{F})$ so that, in
particular, $k(p_{G}-p_{F}) \in T_{G}$.

\begin{lemma}
  \label{lem:adjointmapcondition}
  Let $R$ be a strongly $\bZ^{2}$-graded ring.  For all $k \in \bZ$
  and for all $F,G \in \fS$ with $F \subseteq G$, the
  $\R[T_{F}]$-linear inclusion map
  $\alpha_{F,G} \colon \R[kF + T_{F}] \rTo \R[kG + T_{G}]$ is
  such that its adjoint map
  \begin{equation*}
    \alpha_{F,G}^{\sharp} \colon \R[kF + T_{F}] \tensor_{\R[T_{F}]}
    \R[T_{G}] \rTo \R[kG + T_{G}] \ , \quad r \tensor s \mapsto rs
  \end{equation*}
  is an isomorphism.
\end{lemma}

\begin{proof}
  Choose a partition of unity of type $(kp_{F},-kp_{F})$, say
  $1= \sum_{j} u_jv_j$ with $u_j\in R_{kp_{F}}$ and
  $v_j\in R_{-kp_{F}}$. The map
  \begin{equation*}
    \beta_{F,G} \colon \R[kG + T_{G}] \rTo \R [kF + T_{F}]
    \tensor_{\R[T_{F}]} \R[T_{G}]  \ , \quad x \mapsto \sum_{j} u_{j}
    \tensor v_{j} x
  \end{equation*}
  is well-defined. First, $u_{j}$ has degree
  $kp_{F} \in kF \subseteq kF + T_{F}$ whence
  $u_{j} \in \R [kF + T_{F}]$. Second, every element~$\sigma$ of
  $(kG + T_{G}) \cap \bZ^{2}$ can be written in the form $kp_{G} + \tau$,
  with $\tau \in T_{G} \cap \bZ^{2}$. Thus
  \begin{equation*}
    \sigma - kp_{F} = kp_{G} + \tau - kp_{F} = \tau + k(p_{G} - p_{F}) \in T_{G}
    \ .
  \end{equation*}
  It follows that for $x \in \R[kG + T_{G}]$ the product $v_{j} x$ is
  an element of~$\R[T_{G}]$. --- The map $\beta_{F,G}$ is
  $\R[T_{G}]$-linear and satisfies
  $\alpha_{F,G}^{\sharp} \circ \beta_{F,G} = \id$ by direct
  calculation, using $1= \sum_{j} u_jv_j$. We also have
  \begin{equation*}
    \beta_{F,G} \circ \alpha_{F,G}^{\sharp} (r \tensor s) =
    \beta_{F,G} (rs) = \sum_{j} u_{j}
    \tensor v_{j} rs \underset{(*)}= \sum_{j} u_{j} v_{j} r
    \tensor s = r \tensor s
  \end{equation*}
  where the equality labelled $(*)$ is true since
  $v_{j} r \in \R[T_{F}]$ for any $r \in \R[kF + T_{F}]$, whence
  $\beta_{F,G} \circ \alpha_{F,G}^{\sharp} = \id$.
\end{proof}

\section{\v Cech complexes}

\subsection*{Incidence numbers and \v Cech complexes}

Suppose $P$ is a poset equipped with a strictly increasing ``rank''
function 
\begin{equation*}
  \rk \colon P\rTo \mathbb{N} 
\end{equation*}
and with a notion of ``incidence numbers''
\begin{equation*}
  [ \, \cdot\, , \,\cdot\,] \colon P \times P \rTo \bZ
\end{equation*}
satisfying the following conditions:

\begin{itemize}
\item[(DI1)] $[x\colon y]=0$ unless $x>y$ and $\rk(y)=\rk(x) - 1$.
\item[(DI2)] For all $z<x$ with $\rk(z) = \rk(x)-2$, the set
  \begin{equation*}
    P(z<x)=\{y\in P \,|\, z<y<x\}
  \end{equation*}
  is finite, and
  \[\sum_{y\in P( z<x)} [x\colon y]\cdot[y\colon z]=0 \ .\]
\item[(DI3)] For $x \in P$ with $\rk(x)=1$ the set
  \begin{equation*}
    P(<x)=\{y\in P \,|\, y<x\}
  \end{equation*}
  is finite, and
  \[ \sum_{y \in P(<x)}[y\colon z]=0 \ . \]
\end{itemize}

\noindent We can then associate a ``\textsc{\v Cech} complex'' with a
diagram of \(K\)-modules (\(K\) a unital ring) indexed by~$P$:

\begin{definition}
  Given diagram $\Phi\colon P \rTo K\Mod$ with structure maps
  $\phi_{x,y}\colon \Phi_y\rTo\Phi_x$, define its {\it \textsc{\v
      Cech} complex\/} $\Gamma (\Phi) = \Gamma_{P} (\Phi)$ to be the
  chain complex concentrated in non-positive degrees with chain
  modules
  \begin{equation*}
    \Gamma(\Phi)_{-n} = \bigoplus_{x \in \rk\inv(n)} \Phi_{x} \qquad (n
    \geq 0)
  \end{equation*}
  with differential induced by $[x:y] \phi_{x,y}$ for $x>y$.
\end{definition}

It is easy to check that by virtue of~(DI1) and~(DI2) this is indeed a
chain complex. Condition~(DI3) ensures that a cone
$M \rTo^{\scalebox{2}{.}} \Phi$ on~$\Phi$, that is, a collection of
$K$-module maps $\kappa_{x} \colon M \rTo \Phi_x$ satisfying the
compatibility condition $\kappa_{x}= \phi_{x,y} \circ \kappa_{y}$,
yields a chain complex
\begin{equation*}
  \ldots \lTo \Gamma_{P}(\Phi)_{-1} \rTo \Gamma_{P} (\Phi)_{0}
  \lTo^{\kappa} M
\end{equation*}
with co-augmentation~$\kappa$ induced by the maps $\kappa_{x}$ for $x
\in \rk\inv(0)$.

\subsection*{\v Cech complexes of diagrams of chain complexes}

The assignment $\Phi \mapsto \Gamma_{P} (\Phi)$ is in fact an exact
functor from the category of $P$\nbd-indexed diagrams of
$K$\nbd-modules to the category of chain complexes of
$K$\nbd-modules. Hence it can be prolongated, by levelwise
application, to a functor $\Gamma_{P}$ from the category of
$P$\nbd-indexed diagrams of $K$\nbd-module chain complexes to the
category of double chain complexes of $K$\nbd-modules. The resulting
double complexes has rows $\Gamma_{P}(C_{t})$, that is,
$\Gamma_{P}(C)_{s,t} = \Gamma_{P}(C_{t})_{s}$, and commuting
differentials.  By construction, $\Gamma_{P}(C)_{s,t} = 0$ for
$s > 0$. --- Passage to \textsc{\v Cech} complexes is homotopy
invariant in the following sense:

\begin{lemma}
  \label{qi}
  Let $\chi \colon \Phi \rTo \Psi$ be a natural transformation of
  functors $\Phi,\Psi \colon P\rTo \Ch(K\Mod)$. Suppose that for
  each $p\in P$ the component $\chi_{p} \colon \Phi_{p} \rTo \Psi_{p}$
  is a quasi-isomorphism. Suppose also that the rank function is
  bounded above. Then the induced map of chain complexes
  \begin{equation*}
    \tot \Gamma(\chi) \colon \tot \Gamma (\Phi) \rTo \tot \Gamma
    (\Psi)
  \end{equation*}
  is also a quasi-isomorphism.
\end{lemma}

\begin{proof}
  This is a standard result, we include an elementary proof for
  convenience. For $p \geq 0$ let $\Gamma^{p}$ denote the horizontal
  truncation at~$-p$ of~$\Gamma$. That is, $\Gamma^{p} (\Phi)$ is
  defined to be the double chain complex that agrees with
  $\Gamma(\Phi)$ in columns $-p,\, -p+1,\, \cdots,\, 0$, and is
  trivial otherwise. As the rank function is bounded above,
  $\Gamma^{p} = \Gamma$ for sufficiently large~$p$.

  The obvious surjection $\Gamma^{p+1}(\Phi) \rTo \Gamma^{p} (\Phi)$,
  given by identity maps in horizontal degrees $-p$ and higher, has
  kernel the vertical chain complex
  \begin{equation*}
    K_{p+1}(\Phi) = \bigoplus_{x \in \rk\inv (p+1)} \Phi_x \ ,
  \end{equation*}
  considered as a double chain complex concentrated in
  column~$-(p+1)$. Its totalisation is then the shift suspension
  $K_{p+1}(\Phi)[-p-1]$ of the chain complex~$K$. By induction on
  $p \geq 0$, using the five lemma for ladder diagrams of long exact
  homology sequences induced by the diagram of
  short exact sequences below,
  \begin{diagram}[small]
    0 & \rTo & K_{p+1}(\Phi)[-p-1] & \rTo & \tot \Gamma^{p+1} (\Phi) &
    \rTo & \tot \Gamma^{p} (\Phi) & \rTo & 0 \\
    && \dTo && \dTo && \dTo \\
    0 & \rTo & K_{p+1}(\Psi)[-p-1] & \rTo & \tot \Gamma^{p+1} (\Psi) &
    \rTo & \tot \Gamma^{p} (\Psi) & \rTo & 0
  \end{diagram}
  the induced maps
  $\tot \Gamma^{p} (\chi) \colon \tot \Gamma^{p} (\Phi) \rTo \tot
  \Gamma^{p} (\Psi)$
  are seen to be quasi-iso\-mor\-phisms for all $p \in \bN$.
\end{proof}

\section{Quasi-coherent diagrams}
\label{sec:qcoh_diagram}

As before we denote by the symbol $\fS$ the poset of non-empty faces
of the square $S = [-1,1] \times [-1,1]$. For the purpose of taking
\textsc{\v Cech} complexes we equip $\fS$ with the rank function
$\mathrm{rk}(F) = \dim(F)$, and the usual incidence numbers coming
from the orientations indicated in Fig.~\ref{fig:faces_of_S}. For
example, \([e_{r} : v_{tr}] = 1\) and \([e_{r} : v_{br}] = -1\), and
\([S : e_?] = 1\) for any decoration \(? \in \{t,l,b,r\}\). Let \(R\)
be a \(\bZ^2\)-graded ring.

\begin{definition}
  A {\it quasi-coherent diagram of modules\/} is a functor
  \begin{equation*}
    M \colon \fS \rTo R_{(0,0)} \Mod \ , \quad F \mapsto M^{F}
  \end{equation*}
  as depicted in Fig.~\ref{fig:qc-diag}. In addition, for each~$F$ the
  entry~$M^F$ is to be equipped with a specified structure of an
  $\R[T_{F}]$\nbd-module, extending the given $R_{(0,0)}$-module
  structure. For an inclusion of faces $F \subseteq G$ we require the
  structure map $\alpha^{F,G} \colon M^{F} \rTo M^{G}$ to be
  $\R[T_{F}]$\nbd-linear, such that the adjoint map
  \begin{equation*}
    \alpha^{F,G}_{\sharp} \colon M^{F}\tensor_{\R[T_F]} \R[T_G]\rTo
    M^G \ , \quad m \tensor x \mapsto \alpha^{F,G} (m) \cdot x
  \end{equation*}
  is an isomorphism of $\R[T_{G}]$\nbd-modules.---A {\it
    quasi-coherent diagram of chain complexes\/} is a chain complex of
  quasi-coherent diagrams of modules.
\end{definition}

\begin{figure}[ht]
  \centering
  \begin{diagram}
    M^{v_{tl}} & \rTo^{\alpha^{v_{tl},e_t}} & M^{e_t} & \lTo^{\alpha^{v_{tr},e_t}} & M^{v_{tr}} \\
    \dTo^{\alpha^{v_{tl},e_l}} && \dTo^{\alpha^{e_t,S}} && \dTo^{\alpha^{v_{tr},e_r}} \\
    M^{e_l} & \rTo^{\alpha^{e_l,S}} & M^S & \lTo^{\alpha^{e_r,S}} &M^{e_r} \\
    \uTo^{\alpha^{v_{bl},e_l}} && \uTo^{\alpha^{e_b,S}} && \uTo^{\alpha^{v_{br},e_r}} \\
    M^{v_{bl}}& \rTo^{\alpha^{v_{bl},e_b}} & M^{e_b} &
    \lTo^{\alpha^{v_{br},e_b}} & M^{v_{br}}
  \end{diagram}
  \caption{Quasi-coherent diagram}
  \label{fig:qc-diag}
\end{figure}

We remark that a quasi-coherent diagram of chain complexes can be
considered as a functor defined on~$\fS$ with values in the category
of chain complexes of $R_{(0,0)}$-modules, subject to conditions as above
specified levelwise. Moreover, any quasi-coherent diagram of modules
can be considered as a quasi-coherent diagram of complexes
concentrated in chain degree~$0$.

In case $R$ is a \textsc{Laurent} polynomial ring in two variables
with coefficients in a commutative ring~$K$, a quasi-coherent
diagram of modules is nothing but a quasi-coherent sheaf of modules on
the product $\mathbb{P}^{1}_{K} \times \mathbb{P}^{1}_{K}$ of
the projective line over~$R_{K}$ with itself.

\medbreak

Given a $\bZ^2$\nbd-graded ring~$R$ and $k \in \bZ$, we denote
by~$D(k)$ the diagram depicted in Fig.~\ref{fig:D(k)}.
\begin{figure}[h]
  \centering
  \begin{diagram}
    {\R[kv_{tl} + T_{v_{tl}}] \atop x^{-k} y^{k} \R[x,y\inv]} %
      & \rTo[l>=3em] & {\R[ke_{t} + T_{e_{t}}]  \atop y^{k} \R[x,x\inv,y\inv]}%
      & \lTo[l>=3em] & {\R[kv_{tr} + T_{v_{tr}}] \atop x^{k} y^{k} \R[x\inv,y\inv]} \\
    \dTo && \dTo && \dTo \\
    {\R[ke_{l} + T_{e_{l}}] \atop x^{-k} \R[x,y,y\inv]} %
    & \rTo & {\R[kS + T_{S}]  \atop \R[x,x\inv,y,y\inv]} %
      & \lTo & {\R[ke_{r} + T_{e_{r}}] \atop x^{k} \R[x\inv,y,y\inv]} \\
    \uTo && \uTo && \uTo \\
    {\R[kv_{bl} + T_{v_{bl}}] \atop x^{-k} y^{-k} \R[x,y]}%
    & \rTo & {\R[ke_{b} + T_{e_{b}}] \atop y^{-k} \R[x,x\inv,y]} %
    & \lTo & {\R[kv_{br} + T_{v_{br}}] \atop x^{k} y^{-k} \R[x\inv,y]}
  \end{diagram}
  \caption{The quasi-coherent diagram~$D(k)$}
  \label{fig:D(k)}
\end{figure}
The arrows are inclusion maps. {\it For a strongly
  $\bZ^{2}$\nbd-graded ring this diagram is quasi-coherent\/} as the
adjoint maps are isomorphisms in light of
Lemma~\ref{lem:adjointmapcondition}. These diagrams will play a
central role later on. They are the analogues of certain line bundles
on $\mathbb{P}^{1}_{R_{(0,0)}} \times \mathbb{P}^{1}_{R_{(0,0)}}$,
\viz, the external tensor square of $\mathcal{O}(k)$. We thus are led
to expect the following calculation of its ``global sections'' and
(trivial) ``higher cohomology'' (note
\(H_{-n} \Gamma_{\fS} \iso \lim^{n}\) as shown, for example, in
\cite[Corollary~2.19]{H-cofibres}):

\begin{proposition}
  \label{cechcomplex}
  For $k \geq 0$ the complex
  \begin{equation*}
    \Gamma_{\fS} \big( D(k) \big) \lTo^{\iota} \bigoplus_{x,y =
      -k}^{k} R_{(x,y)} \lTo 0
  \end{equation*}
  is exact, where $\iota$ is the diagonal embedding, induced from
  the inclusions of its source into $D(k)^{v}$ with $v$ a vertex
  of~$S$.

  More explicitly, the sequence
  \begin{multline}
    \label{eq:Cech_Dk}
    0 \lTo  \Gamma_{\fS} \big( D(k) \big)_{-2} \lTo^{d_{-1}} %
     \Gamma_{\fS} \big( D(k) \big)_{-1} \\ %
    \lTo^{d_{0}} \Gamma_{\fS} \big( D(k) \big)_{0} \lTo^{\iota} %
    \bigoplus_{x,y = -k}^{k} R_{(x,y)} \lTo 0
  \end{multline}
  is exact.
\end{proposition}

\begin{proof}
  The complex consists of $\bZ^{2}$\nbd-graded $R_{(0,0)}$-modules and
  degree-pre\-ser\-ving maps, so exactness can be checked in each
  degree separately. In degree $(x,y) \in [-k,k]^2$, the complex is
  the dual of the augmented cellular chain complex of the square
  (equipped with its obvious cellular structure), which is
  contractible, tensored with~$R_{(x,y)}$. For all other degrees
  $(x,y)$, we see the dual of the cellular complex associated to the
  complement of a visibility subcomplex of the boundary of~$S$. In
  either case, the resulting complex is acyclic. --- More details on
  the computation can be found in the appendix of~\S2.5
  in~\cite{MR2494885}, for example.
\end{proof}

\section{Extending chain complexes of modules to quasi-coherent diagrams of chain complexes}

Suppose that $M$ is a finitely generated free module over the
$\bZ^{2}$\nbd-graded ring~$R$. Then there exists a quasi-coherent
diagram which has $M$ as its middle entry. More precisely, we fix an
isomorphism between $M$ and $R^{n}$, for some $n \geq 0$, and have
$\big( \bigoplus_{n} D(k) \big)^{S} \iso M$ for any $k \in \bZ$. This
shows that finitely generated free $R$-modules can be extended to
quasi-coherent diagrams. We need the following chain complex version
of this fact:

\begin{proposition}
  \label{formingasheaf}
  Let $R$ be a $\bZ^{2}$\nbd-graded ring, and let $C$ be a bounded
  above chain complex of finitely generated free $R$-modules. Then
  there exists a complex~$\cY$ of diagrams of the form
  $\bigoplus_{r} D(k)$, for various $r \geq 0$, such that
  $C \iso \cY^{S}$ as $R$\nbd-module complexes. More precisely,
  suppose that $C$ is concentrated in degrees $\leq t$. There are
  numbers
  \begin{equation*}
    0 \leq k_{t} \leq k_{t-1} \leq k_{t-2} \leq \ldots
  \end{equation*}
  and numbers $r_{n} \geq 0$, with $r_{n} = 0$ whenever $C_{n} = 0$,
  such that we can choose $\cY_{n} = \bigoplus_{r_{n}} D(k_{n})$. If
  $C$ is bounded, then so is~$\cY$. --- If the ring~$R$ is strongly
  $\bZ^{2}$\nbd-graded then $\cY$ consists of quasi-coherent diagrams.
\end{proposition}

\begin{proof}
  We identify the chain module $C_{n}$ with $R^{r_{n}}$, for suitable
  numbers $r_{n} \geq 0$, where $r_{n} = 0$ if and only if
  $C_{n} = 0$; this amounts to choosing basis elements for the free
  modules~\(C_{n}\). The action of the differential
  $C_{n} \rTo C_{n-1}$ is then given by left multiplication by a
  matrix~$D_{n}$ of size $r_{n-1} \times r_{n}$, with entries in~$R$.

  Given a homogeneous non-zero element $x \in R_{(s,t)}$ we write
  $a(x) = \max \big( |s|, |t| \big)$; we also agree $a(0) = 0$. For a
  general element $x \in R$ let $a(x)$ denote the maximum of the
  values $a(x_i)$, where the $x_i$ are the homogeneous components
  of~$x$. We denote by $a_{n}$ the maximum of the values $a(x)$ where
  $x$ varies over the entries of the matrix~$D_{n}$, with the
  convention that $a_{n} = 0$ if $D_{n}$ is the empty matrix (\ie, if
  $r_{n} = 0$ or $r_{n-1} = 0$).

  Suppose that $C$ is concentrated in chain degrees~$\leq t$. We set
  $k_{n} = 0$ for $n \geq t$, and
  \begin{equation*}
    k_{n} = \sum_{j=n+1}^{t} a_{j}
  \end{equation*}
  for $k < t$. We then define the chain complex~$\cY$ of
  quasi-coherent diagrams by setting
  $\cY_{n} = \bigoplus_{r_{n}} D(k_{n})$; the differentials
  $d_{n} \colon \cY_{n} \rTo \cY_{n-1}$ are given by left
  multiplication by the matrix $D_{n}$ in each component, so that
  \begin{equation*}
    d_{n}^{F} \colon \cY_{n}^{F} = \big( \bigoplus_{r_{n}} D(k_{n}) \big)^{F}
    \rTo^{\cdot D_{n}} \big( \bigoplus_{r_{n-1}} D(k_{n-1}) \big)^{F} =
    \cY_{n-1}^{F} \ .
  \end{equation*}
  (This is well defined: The numbers \(k_{n}\) are chosen precisely to
  ensure that \(d_{n}^{F}\) maps \(\cY_{n}^{F}\) to~\(\cY_{n-1}^{F}\).)
  As $D_{n-1} \cdot D_{n} = 0$ we have $d_{n-1} \circ d_{n} = 0$, and
  $C = \cY^{S}$ by construction.

  It was observed before that the diagrams $D(k)$ are quasi-coherent
  if $R$ is strongly $\bZ^{2}$-graded (this follows from
  Lemma~\ref{lem:adjointmapcondition}); hence $\cY$ consists of
  quasi-coherent diagrams in this case.
\end{proof}

\begin{corollary}
  \label{findomiso}
  Let $R$, $C$, $\cY$, $r_{n}$ and $k_{n}$ be as in
  Proposition~\ref{formingasheaf}, with $C$ bounded. Then
  $D = \lim \cY$ is a bounded complex of \(R_{(0,0)}\)-modules with
  chain modules
  \begin{equation*}
    D_{n} =  \Big( \bigoplus_{x,y
      =-k_{n}}^{k_{n}} R_{(x,y)} \Big)^{r_{n}} \ .
  \end{equation*}
  If $R$ is strongly $\bZ^{2}$\nbd-graded these $R_{(0,0)}$-modules
  are finitely generated projective. The natural map
  $D \rTo \tot \, \Gamma_{\fS} (\cY)$ is a homotopy equivalence of
  $R_{(0,0)}$\nbd-module complexes so that
  $\tot \, \Gamma_{\fS} (\cY)$ is $R_{(0,0)}$\nbd-finitely dominated.
\end{corollary}

\begin{proof}
  Consider the double complex concentrated in columns \(-2\), \(-1\),
  \(0\) and~\(1\) with \(n\)th row given by
    \begin{equation*}
    \Gamma_{\fS} \big( \cY_{n} \big) \lTo^{\iota} \Big( \bigoplus_{x,y =
      -k_{n}}^{k_{n}} R_{(x,y)} \Big)^{r_{n}}\ .
  \end{equation*}
  Column~1 is the chain complex \(\lim\, \cY\), see
  Proposition~\ref{cechcomplex}, which consists of finitely generated
  projective \(R_{(0,0)}\)-modules if \(R\) is strongly graded by
  Corollary~\ref{cor:components_fgp}.  As the rows are exact (by
  Proposition~\ref{cechcomplex} again), standard homological algebra
  asserts that the natural map \(\lim \cY \rTo \tot \, \Gamma(\cY)\)
  is a quasi-isomorphism. In the strongly graded case, this is a
  homotopy equivalence (all complexes are bounded below and consist of
  projective \(R_{(0,0)}\)-modules) so that \(\tot \, \Gamma(\cY)\) is
  \(R_{(0,0)}\)-finitely dominated.
\end{proof}

\section{Flags and stars, and their associated rings}

\begin{definition}
  For $F \in \fS$ define the \emph{star of $F$},
  denoted~$\mathrm{st}(F)$, as the set of faces of~$S$ that
  contain~$F$; this is a sub-poset of~$\fS$. Let $\cN_{F}$ be the set
  of flags in~$\mathrm{st}(F)$; the elements of $\cN_{F}$ are strictly
  increasing sequences
  $\tau = (P_0\subset P_1 \subset ... \subset P_k)$ of faces of~$S$
  with $F \subseteq P_{0}$. We consider $\cN_{F}$ as a poset with
  order given by inclusion (or refinement) of flags, the smaller flag
  having fewer elements.
\end{definition}

For each $F \in \fS$ we equip $\cN_{F}$ with the rank function
\begin{equation*}
  \rk (P_0\subset P_1 \subset ... \subset P_k) = k \ .
\end{equation*}
and with standard simplicial incidence numbers: $[A:B] = 0$ unless the
flag~$B$ is obtained from
$A = (P_0\subset P_1 \subset ... \subset P_k)$ by omitting the
entry~$P_{j}$, in which case $[A:B] = (-1)^j$. Whenever we
form a \textsc{\v Cech} complex of a diagram indexed by~$\cN_{F}$ we
will use these data.

\medbreak

We now attach a ring $A_{\tau}$ to each flag
$\tau = (F_0 \subset \cdots \subset F_k)$ of faces in~$\fS$, as
follows:

\begin{equation*}
  \begin{array}{lcl}
    \lr{v_{bl}}         & = & \R\powers {x,y}         \\ 
    \lr{e_b}            & = & \R[x,x\inv] \powers {y} \\
    \lr{S}              & = & \R[x,x\inv, y,y\inv] = R   \\
    \lr{v_{bl}, e_b}    & = & \R\nov {x} \powers{y}   \\
    \lr{v_{bl}, S}      & = & \R\nov {x, y}     \\ 
    \lr{e_b, S}         & = & \R[x,x\inv]\nov {y}     \\
    \lr{v_{bl}, e_b, S} & = & \R\nov {x} \nov{y}
  \end{array}
\end{equation*}
The effect of replacing ``left'' by ``right'' (\ie, replacing the
subscript ``$l$'' by~''$r$'' throughout) is to replace~$x$ by~$x\inv$,
while replacing ``bottom'' by~''top'' (\ie, replacing the subscript ``$b$''
by~''$t$'' throughout) means replacing~$y$ by~$y\inv$. Swapping
``bottom'' and ``left'' amounts to swapping $x$ with~$y$, and swapping
``top'' and~''right'' results in $x\inv$ and~$y\inv$ swapping
places. Thus, for example,
\begin{equation*}
  \lr{v_{tl},e_l}  = \R\nov {y\inv} \powers{x} %
  \quad \text{and} \quad %
  \lr{v_{tr}, e_r, S} = \R\nov {y\inv} \nov{x\inv} \ .
\end{equation*}
By direct inspection, this collection of rings is seen to have the
following properties:
\begin{enumerate}
\item If $\sigma \subseteq \tau$ then
  $A \langle \sigma \rangle \subseteq A \langle \tau \rangle$. In
  particular, if $S \in
  \tau$ 
  then $R = \R[x,x\inv, y,y\inv] = R \langle S \rangle$ is a subring
  of~$A \langle \tau \rangle$.
\item If $\tau \in \cN_{F}$ (that is, if all faces in~$\tau$
  contain~$F$) then $A_{F} \subseteq A \langle \tau \rangle$.
\end{enumerate}

\medbreak

Fix $F \in \fS$. The rings $\lr{\tau}$ defined above fit into a
commutative diagram
\begin{equation*}
  E_F\colon \cN_F\rTo A_{F}\Mod \ ,\quad \tau
  \mapsto \lr{\tau}
\end{equation*}
of $A_{F}$-modules with structure maps given by inclusions. For
$F = S$, the diagram $E_{S}$ has a single entry indexed by the flag
$\{S\}$, and takes value $\lr{S} = R$ there. For $F = e_{l}$, the
left-hand edge of~$S$, the diagram $E_{e_{l}}$ looks like this:
\begin{equation}
\label{eq:E_el}
  {\R [y, y\inv] \powers{x} \atop \lr{e_l}} %
  \rTo[l>=3em] \relax %
  {\R [y, y\inv] \nov{x} \atop \lr{e_{l},S}} %
  \lTo[l>=3em] \relax %
  {\R [x, x\inv, y, y\inv] \atop \lr{S}}
\end{equation}
And finally, for $F = v_{bl}$ the bottom left vertex of~$S$, the
diagram $E_{v_{bl}}$ is depicted in Fig.~\ref{fig:Evbl}.

\begin{figure}[ht]
  \centering
  \begin{diagram}[small]
    {\lr{e_l} \atop \R[y,y\inv] \powers{x}} && \rTo && {\lr{e_l,S}
      \atop \R[y,y\inv] \nov{x} }&& \lTo && {\lr{S} \atop \R[x,x\inv,y,y\inv]}
    \\ %
    &&& \ldTo & &&& \ldTo(4,4)\\ %
    \dTo && {\lr{v_{bl}, e_{l}, S} \atop \R \nov{y} \nov{x}} && &&&& \dTo \\
    & \ruTo && \luTo \\%
    {\lr{v_{bl}, e_{l}} \atop \R \nov{y} \powers{x}} &&&& {\lr{v_{bl},
        S} \atop \R \nov{x,y}} &&&& {\lr{e_{b}, S} \atop \R[x,x\inv]
      \nov{y}} \\
    &&& \ruTo(4,4) & & \rdTo && \ldTo \\
    \uTo &&&& && {\lr{v_{bl}, e_{b}, S} \atop \R \nov{x} \nov{y}} && \uTo \\
    &&&& & \ruTo \\
    {\lr{v_{bl}} \atop \R \powers{x,y}} && \rTo && {\lr{v_{bl}, e_{b}}
      \atop \R \nov{x} \powers{y}} && \lTo && {\lr{e_{b}} \atop
      \R[x,x\inv] \powers{y}}
  \end{diagram}
  \caption{The diagram~$E_{v_{bl}}$}
  \label{fig:Evbl}
\end{figure}

We let $\Gamma_{\cN_{F}} (E_{F})$ denote the \textsc{\v Cech} complex
of~$E_{F}$. Considering $A_{F}$ as a chain complex concentrated in
chain level~$0$ we have a chain map
\begin{equation*}
  \xi^{F} \colon A_{F} \rTo \Gamma_{\cN_{F}} (E_{F})
\end{equation*}
given by the diagonal embedding
$A_{F} \rTo \bigoplus_{G \supseteq F} \lr{G}$. This is indeed a chain
map thanks to property~(DI3) of incidence numbers.

\begin{lemma}
  \label{lem:xiF_is_qi}
  The map $\xi^{F}$ is a quasi-isomorphism of chain complexes of
  $A_{F}$-modules.
\end{lemma}

\begin{proof}
  Note that for $F = S$ there is actually nothing to show as the
  category $\cN_{S}$ has a single object; for $F = e_{l}$ the claim is
  equivalent to saying that the sequence
  \begin{multline*}
    0 \rTo \R[x,y,y\inv]
    \rTo^{\left(\begin{smallmatrix}1\\1\end{smallmatrix}\right)}
    \R[y,y\inv] \powers{x} \oplus \R[x,x\inv,y,y\inv] \\
    \noalign{\smallskip}
    \rTo^{(1 \ -1)} \R[y,y\inv] \nov{x} \rTo 0
  \end{multline*}
  is exact, and it is not too hard to see that it is actually split
  exact as a sequence of \(R_{(0,0)}\)-modules. The case of
  \(F = v_?\) a vertex is somewhat more demanding in terms of
  book-keeping.  Full details are worked out in the proof of
  Lemma~4.5.1 in~\cite{TWOV}, writing \(\R[x,x\inv]\nov{y}\) in place
  of \(R[x,x\inv]\nov{y}\) (and similar for the other rings); the
  proof carries over \textit{mutatis mutandis} since the claim can be
  checked in each degree separately.
\end{proof}

For faces $G \supsetneqq F$ let $E_{F}^{G}$ denote the
$\cN_{F}$\nbd-indexed diagram which agrees with $E_{G}$ on~$\cN_{G}$,
and is zero everywhere else. The obvious map of diagrams
$E_{F} \rTo E_{F}^{G}$, given by identities where possible, results in
a map of chain complexes
\begin{equation*}
  \Gamma_{\cN_{F}} (E_{F}) \rTo \Gamma_{\cN_{F}} (E_{F}^{G}) =
  \Gamma_{\cN_{G}} (E_{G}) \ ,
\end{equation*}
given by projecting away from all those summands in the source which
involve flags in $\cN_{F} \setminus \cN_{G}$. These maps are
functorial on the poset~$\fS$ so that there results a commutative
diagram
\begin{equation*}
  E \colon \fS \rTo R_{(0,0)}\Mod \ , \quad F \mapsto \Gamma_{\cN_{F}}
  (E_{F}) \ .
\end{equation*}

\begin{proposition}
\label{prop:xi_qi}
The maps~$\xi^{F}$ constructed above assemble to a natural
transformation
\begin{equation*}
  \xi \colon D(0) \rTo E \ .
\end{equation*}
Its components $\xi^{F} \colon A_{F} \rTo \Gamma_{\cN_{F}} (E_{F})$
are quasi-isomorphisms of complexes of $A_{F}$\nbd-modules.
\end{proposition}

\begin{proof}
  This is the content of Lemma~\ref{lem:xiF_is_qi}, together with the
  observation that the maps~$\xi^{F}$ are natural in~$F$ with respect
  to the structure maps in~$E$.
\end{proof}

Given a flat $A_{F}$-module $M_{F}$ we thus obtain a quasi-isomorphism
\begin{equation*}
  \chi^{F} \colon M_{F} \rTo^{\iso} M_{F} \tensor_{A_{F}} A_{F} %
  \rTo^{\id \tensor \xi^{F}} M_{F} \tensor_{A_{F}} \Gamma_{\cN_{F}}(E_{F}) %
  \iso \Gamma_{\cN_{F}} (M_{F} \tensor_{A_{F}} E_{F}) \ ;
\end{equation*}
here $M_{F} \tensor_{A_{F}} E_{F}$ stands for the pointwise tensor
product of the module~$M_{F}$ with the entries of the
diagram~$E_{F}$. In other words, the sequence
\begin{multline}
  \label{eq:exact_sequence}
  0 \rTo M_{F} \rTo \Gamma_{\cN_{F}} (M_{F} \tensor_{A_{F}} E_{F})_{0}
  \rTo \Gamma_{\cN_{F}} (M_{F} \tensor_{A_{F}} E_{F})_{-1} \\
  \rTo \Gamma_{\cN_{F}} (M_{F} \tensor_{A_{F}} E_{F})_{-2} \rTo 0
\end{multline}
is exact.

For $M$ an arbitrary quasi-coherent diagram we obtain a natural
transformation
\begin{equation}
  \label{eq:nat_chi}
  \chi \colon M \rTo^{\iso} M \tensor_{D(0)} D(0) \rTo M
  \tensor_{D(0)} E \rTo^{\iso} M'
\end{equation}
where the target $M'$ is the diagram of chain complexes
\begin{equation*}
  M' \colon F \mapsto \Gamma_{\cN_{F}} (M_{F} \tensor_{A_{F}} E_{F}) \ ,
\end{equation*}
$M \tensor_{D(0)} D(0)$ denotes the pointwise tensor product
\begin{equation*}
  M \tensor_{D(0)} D(0) \colon F \mapsto M_{F} \tensor_{A_{F}}
  D(0)^{F} = M_{F} \tensor_{A_{F}} A_{F} \ ,
\end{equation*}
and $M \tensor_{D(0)} E$ stands similarly for the pointwise tensor
product of the diagrams~$M$ and~$E$.  If moreover $M$ is such that the
entry~$M_{F}$ is a flat $A_{F}$-module, the components of~$\chi$ are
quasi-isomorphisms.

\section{From trivial \textsc{Novikov} homology to finite domination}

We will now prove the ``if'' implication of Theorem~\ref{thm:main}. So
let $R$ be a strongly $\bZ^{2}$\nbd-graded ring. Suppose that $C$ is a
bounded complex of finitely generated free $R$-modules, and further
that the complexes listed in~\eqref{eq:cond_edge_new}
and~\eqref{eq:cond_vertex_new} are acyclic. In view of the assumed
freeness of~$C$, these eight complexes are then actually
contractible. They are of the form
\begin{equation*}
  C \tensor_{R} \lr{e, S} 
  \qquad \text{and} \qquad %
  C \tensor_{R} \lr{v, S} 
\end{equation*}
where $e$ and~$v$ denote an edge and a vertex of~$S$, respectively. As
tensor products preserve contractions, it follows that for all eight
maximal flags $v \subset e \subset S$, for $v$ a vertex of~$S$ and $e$
an edge incident to~$v$, the complex
\begin{equation}
  \label{eq:is_contractible_as_well}
  C \tensor_{R} \lr{v, e, S} \iso %
  C \tensor_{R} \lr{v, S} \tensor_{\lr{v, S}} \lr{v, e, S} \simeq 0
\end{equation}
is contractible (and hence acyclic) as well.

Extend the complex~$C$ to a complex of sheaves
$\cY \colon F \mapsto \cY_{F}$, which can be done by
Proposition~\ref{formingasheaf}. In more detail, this means that we
find a complex of sheaves such that $C$ is identified with $\cY^{S}$,
and such that there is a quasi-isomorphism
$\lim \cY \rTo \tot \, \Gamma_{\fS} (\cY)$ with $R_{(0,0)}$\nbd-finitely
dominated source (Corollary~\ref{findomiso}).

Let, for the moment, $F \neq \emptyset$ denote a fixed face of~$S$. We
observe that if $F \neq S$,
\begin{equation}
  \label{eq:contractible_1}
  \cY_{F} \tensor_{A_{F}} \lr {F,S} \iso \cY_{F} \tensor_{A_F} A_{S}
  \tensor_{A_{S}} \lr {F, S} \underset{(\dagger)}{\iso} C \tensor_{R} \lr
  {F,S} \simeq 0 \ . 
\end{equation}
The isomorphism labelled~($\dagger$) combines two facts: first, the entries
of the complex~$\cY$ are quasi-coherent diagrams
(Proposition~\ref{formingasheaf}) so that
$\cY_{F} \tensor_{A_F} A_{S} \iso \cY^{S}$; second, by construction
of~$\cY$ there is an identification of~$\cY^{S}$ with~$C$.

If $F = v$ is a vertex and $e \supseteq v$ an edge incident to~$e$, we
make use of \eqref{eq:is_contractible_as_well} to conclude similarly
\begin{equation}
  \label{eq:contractible_2}
  \cY_{v} \tensor_{A_{v}} \lr {v,e,S} \iso \cY_{v} \tensor_{A_v} A_{S}
  \tensor_{A_{S}} \lr {v, e, S} \iso C \tensor_{R} \lr
  {v, e, S} \simeq 0 \ . 
\end{equation}

For arbitrary~$F \neq \emptyset$, exactness of the
sequence~\eqref{eq:exact_sequence} implies that the double complex
\begin{equation*}
  \cY_{F} \rTo \Gamma_{\cN_{F}} ( \cY_{F} \tensor_{A_{F}} E_{F})
\end{equation*}
has exact rows (using the fact that $\cY_{F}$ consists of projective
$A_{F}$\nbd-modules), so there results a quasi-isomorphism
\begin{equation*}
  \tot(\chi^{F}) \colon \cY_{F} \rTo \tot \, \Gamma_{\cN_{F}} (
  \cY_{F} \tensor_{A_{F}} E_{F}) \ .
\end{equation*}

We will now make use of the fact that some entries of the diagram
$\cY_{F} \tensor_{A_{F}} E_{F}$ are known to be acyclic. Let $Z_{F}$
denote the $\cN_{F}$\nbd-indexed diagram which agrees with
$\cY_{F} \tensor_{A_{F}} E_{F}$ on those flags not containing~$S$, and
is zero elsewhere. Let $K_{F}$ denote the diagram which takes the
value~$C$ at~$\{S\}$, and is zero otherwise.  The obvious surjective
map $\cY_{F} \tensor_{A_{F}} E_{F} \rTo Z_{F} \oplus K_{F}$ is a
pointwise quasi-isomorphism, by \eqref{eq:contractible_1}
and~\eqref{eq:contractible_2}, and hence induces a quasi-isomorphism
\begin{equation*}
  \tot \, \Gamma_{\cN_{F}} ( \cY_{F} \tensor_{A_{F}} E_{F})
  \rTo^{\simeq} \tot \, \Gamma_{\cN_{F}} (Z_{F} \oplus K_{F})  \ .
\end{equation*}
We compute further that the target of this map is
\begin{equation*}
  \tot \, \Gamma_{\cN_{F}} (Z_{F}) \oplus \tot \Gamma_{\cN_{F}}
  (K_{F}) = \tot \, \Gamma_{\cN_{F}} (Z_{F}) \oplus C \ .
\end{equation*}
In total this yields a quasi-isomorphism
\begin{equation*}
  \Xi^{F} \colon \cY_{F} \rTo \tot \, \Gamma_{\cN_{F}} (Z_{F}) \oplus C \ .
\end{equation*}

Allowing $F$ to vary again, and making use of the naturality of the
constructions above, we see that we obtain a natural quasi-isomorphism
of diagrams
\begin{equation*}
  \Xi \colon \cY \rTo \tot \, \Gamma_{\cN_{(\,-\,)}} (Z_{(\,-\,)}) \oplus
  \mathrm{con}\,(C) \ ,
\end{equation*}
with the $\fS$\nbd-indexed diagram
\begin{gather*}
  \Gamma_{\cN_{(\,-\,)}} (Z_{(\,-\,)}) \colon F \mapsto \Gamma_{\cN_{F}}
  (Z_{F}) \\ %
  \intertext{and the constant diagram}
  \mathrm{con}\,(C) \colon F \mapsto C \ .
\end{gather*}
It follows that $\tot \, \Gamma_{\fS} (\cY)$, which is an
$R_{(0,0)}$\nbd-finitely dominated by Corollary~\ref{findomiso}, is
quasi-isomorphic to
\begin{equation*}
  \tot \, \Gamma_{\fS} \big( \tot \, \Gamma_{\cN_{(\,-\,)}} (Z_{(\,-\,)})
  \big) \oplus \tot \, \Gamma_{\fS} \big( \mathrm{con}\,(C) \big) \ .
\end{equation*}
The second summand is, in turn, quasi-isomorphic to~$C$. (In effect,
this is true since the nerve of~$\fS$ is contractible; for a more
explicit argument, observe that the double complex
\begin{equation*}
  \Gamma_{\fS} \big( \mathrm{con}\,(C) \big) \lTo C
\end{equation*}
concentrated in columns \(-2\), \(-1\), \(0\) and~\(1\) has acyclic
rows so that its totalisation is acyclic. But the totalisation is, up
to isomorphism, the mapping cone of the map
\(C \rTo \tot \, \Gamma_{\fS} \big( \mathrm{con}\,(C) \big)\) which is
thus a weak equivalence. See also \cite[Lemma~4.6.4]{TWOV}.)

Thus in the derived category of the ring~$R_{(0,0)}$ the complex $C$ is a
retract of $\tot \, \Gamma_{\fS} (\cY)$, and as both complexes are
bounded and consist of projective $R_{(0,0)}$\nbd-modules, this makes $C$
a retract up to homotopy of~$\tot \, \Gamma_{\fS} (\cY)$. As the
latter is $R_{(0,0)}$\nbd-finitely dominated so is~$C$, as was to be
shown.

\part{Finite domination implies triviality of \textsc{Novikov} homology.}
\label{part:finite-domin-impl}

\section{Algebraic tori}

\begin{definition}
  Let $R = \bigoplus_{\sigma \in \bZ^{2}} R_{\sigma}$ be a strongly
  $\bZ^2$-graded unital ring. Given an $R$-module~$M$ and an element
  $\rho\in\bZ^2$, we define the map of right $R$\nbd-modules
  \begin{equation*}
    \chi_\rho \colon M \tensor_{R_{(0,0)}} R \rTo %
    M \tensor_{R_{(0,0)}} R \ , \quad %
    m \tensor r \mapsto \sum_{j} m{u_j}\tensor{v_j}r \ ,
  \end{equation*}
  where $1 = \sum_{j} u_{j} v_{j}$ is any partition of unity of
  type~$(-\rho, \rho)$.
\end{definition}

The map $\chi_\rho$ is $R_{(0,0)}$\nbd-balanced and independent of the
choice of partition of unity as its restriction to $M
\tensor_{R_{(0,0)}} R_{\sigma}$ can be re-written as
\begin{equation*}
  M \tensor_{R_{(0,0)}} R_{\sigma} \rTo^{\iso}_{\mu_{-\rho,
      \sigma+\rho}} M \tensor_{R_{(0,0)}} 
  R_{-\rho} \tensor_{R_{(0,0)}} R_{\sigma+\rho} \rTo_{\omega} M
  \tensor_{R_{(0,0)}} R
\end{equation*}
with $\mu_{-\rho, \sigma+\rho}$ the $R_{(0,0)}$\nbd-bimodule
isomorphism from Proposition~\ref{prop:pi_mu}, and with
$\omega (m \tensor x \tensor y) = (mx) \tensor y$. --- In case $M=R$,
the map $\chi_{\rho}$ is an $R$\nbd-bimodule homomorphism.

Of course $\chi_{(0,0)} = \id$. As $\chi_{\rho}$ does not depend on
the choice of partition of unity, Lemma~\ref{lem:new_pou_from_old}
implies:

\begin{corollary}
  \label{cor:chi-commutes}
  For all $\rho, \sigma \in \bZ^{2}$ there are equalities of maps
  $\chi_\rho \chi_\sigma = \chi_{\rho+\sigma} = \chi_{\sigma}
  \chi_{\rho}$. The maps $\chi_{\rho}$ are isomorphisms of
  $R$\nbd-modules with $\chi_{\rho}\inv = \chi_{-\rho}$. \qed
\end{corollary}

We can consider
\begin{equation*}
  M \tensor_{R_{(0,0)}} R = \bigoplus_{\rho \in \bZ^2} M
  \tensor_{R_{(0,0)}} R_{\rho}
\end{equation*}
as a \(\bZ^2\)-graded \(R\)-module. With respect to this grading, we
observe:

\begin{lemma}
  \label{lem:chi_homogeneous}
  The map \(\chi_{\rho}\) maps homogeneous elements of
  degree~\(\sigma\) to homogeneous elements of degree \(\rho +
  \sigma\). For the homogeneous element \(m \tensor r\) of
  degree~\(-\rho\), the formula \(\chi_{\rho} (m \tensor r) = mr
  \tensor 1\) holds. \qed
\end{lemma}

Let \(C\) be a chain complex of \(R\)-modules. Then \(\chi_{\rho}\),
applied in each chain level, defines a chain map
\begin{equation*}
  \chi_{\rho} \colon C \tensor_{R_{(0,0)}} R \rTo C
  \tensor_{R_{(0,0)}} R \ .
\end{equation*}
Let \(D\) be an additional \(R_{(0,0)}\)-chain complex, let
$\alpha \colon C \rTo D$ and $\beta \colon D\rTo C$ be
\(R_{(0,0)}\)-linear chain maps, and let
$H\colon\id_C\simeq \beta\alpha$ be a homotopy from \(\beta\alpha\)
to~\(\id\) such that $dH+Hd=\beta\alpha-\id_C$. Then the outer square
in the diagram \(T(\alpha, \beta; H)\)
\begin{equation*}
  \label{diag:T}
  \begin{diagram}[labelstyle=\scriptstyle]
    D \tensor_{R_{(0,0)}} R                & \rTo^{\scriptstyle
      \id^*-\alpha^*\chi_{e_1}\beta^*}     & D \tensor_{R_{(0,0)}} R                                                                                                           \\
    \dTo^{\id^*-\alpha^*\chi_{e_2}\beta^*} & \rdTo^{\alpha^*\chi_{e_1}H^*\chi_{e_2}\beta^*}_{-\alpha^*\chi_{e_2}H^*\chi_{e_1}\beta^*} & \dTo_{\id^*-\alpha^*\chi_{e_2}\beta^*} \\
    D \tensor_{R_{(0,0)}} R                & \rTo_{\id^*-\alpha^*\chi_{e_1}\beta^*}                                                   & 
    D \tensor_{R_{(0,0)}} R
  \end{diagram}
  \tag*{\(T(\alpha, \beta; H)\)}
\end{equation*}
(writing \(f^* = f \tensor \id\), for any map~\(f\)) is homotopy
commutative, with the diagonal arrow recording a preferred homotopy of
the two possible compositions. --- Note that if $\beta\alpha = \id_C$
we can choose $H=0$ and in this case the outer squares of
$T(\alpha, \beta; 0)$ commutes by Corollary~\ref{cor:chi-commutes}.

\medskip

\begin{definition}
  The \textit{algebraic torus} $\torb{\alpha}{\beta}{H}$ is defined as
  the totalisation of the homotopy commutative
  diagram~\(T(\alpha, \beta; H)\). That is,
  \(\torb{\alpha}{\beta}{H}\) is the complex with chain modules
  \begin{multline*}
    \torb{\alpha}{\beta}{H}_n = \\
    (D_{n-2}\tensor_{R_{(0,0)}} R )\oplus (D_{n-1}\tensor_{R_{(0,0)}} R) \oplus (D_{n-1}\tensor_{R_{(0,0)}} R )
    \oplus (D_{n}\tensor_{R_{(0,0)}} R)
  \end{multline*}
  and boundary $d_{\torb{\alpha}{\beta}{H}}$ given by the following matrix:
  \begin{align*}
    \begin{pmatrix}
      d^* & 0 & 0 & 0 \\
      \id^*-\alpha^*\chi_{e_1}\beta^* & -d^* & 0 & 0 \\
      \id^*-\alpha^*\chi_{e_2}\beta^* & 0 & -d^* & 0 \\
      \alpha^*(\chi_{e_1}H^*\chi_{e_2} - \chi_{e_2}H^*\chi_{e_1})\beta^* & \id^*-\alpha^*\chi_{e_2}\beta^* & -\id^*+\alpha^*\chi_{e_1}\beta^* & d^* \\
    \end{pmatrix}
  \end{align*}
\end{definition}

\section{Canonical resolutions}

Let \(C\) be a chain complex of \(R\)-modules.
The commutative diagram
\begin{equation*}
  \begin{diagram}[labelstyle=\scriptstyle]
    C \tensor_{R_{(0,0)}} R & \rTo^{\id^*-\chi_{e_1}} & C \tensor_{R_{(0,0)}} R \\
    \dTo^{\id^*-\chi_{e_2}} && \dTo_{\id^*-\chi_{e_2}} \\
    C \tensor_{R_{(0,0)}} R & \rTo^{\id^*-\chi_{e_1}} & C \tensor_{R_{(0,0)}} R
  \end{diagram}
  \tag*{\(T(\id_{C}, \id_{C}; 0)\)}
\end{equation*}
gives rise, \textit{via} totalisation, to the algebraic torus
\(\torb{\id_{C}}{\id_{C}}{0}\). The \(R\)-module structure map
\(\gamma \colon (x,r) \mapsto xr\) induces a map of \(R\)-module chain
complexes
\begin{equation*}
  \kappa \colon \torb{\id_{C}}{\id_{C}}{0} \rTo C \ .
\end{equation*}

\begin{proposition}[Canonical resolution]
  \label{prop:canonical-resolution}
  The map \(\kappa\) is a quasi-isomorphism. If \(C\) is a bounded
  below complex of projective \(R\)-modules then \(\kappa\) is a chain
  homotopy equivalence.
\end{proposition}

In preparation of the proof, we note that
\(\torb{\id_{C}}{\id_{C}}{0}\) can be described as the totalisation
(in the usual sense) of a double complex of the form
\begin{equation}
  \label{eq:can-res-R}
  C \tensor_{R_{(0,0)}} R \rTo %
  \Big( C \tensor_{R_{(0,0)}} R \Big)^2 \rTo %
  C \tensor_{R_{(0,0)}} R \ .
\end{equation}
The \(R\)-module structure map \(\gamma \colon (x,r) \mapsto xr\)
of~\(C\) can be used to augment this to the \(R\)-module double complex
\begin{equation}
  \label{eq:can-res-R-augmented}
  0 \rTo C \tensor_{R_{(0,0)}} R \rTo %
  \Big( C \tensor_{R_{(0,0)}} R \Big)^2 \rTo %
  C \tensor_{R_{(0,0)}} R \rTo^{\gamma} C \rTo 0 \ .
\end{equation}
We will consider a single row of this double complex, that is, a
sequence of the form
\begin{equation}
  \label{eq:can-res-row}
  0 \rTo M \tensor_{R_{(0,0)}} R \rTo^{\alpha} %
  \Big( M \tensor_{R_{(0,0)}} R \Big)^2 \rTo^{\beta} %
  M \tensor_{R_{(0,0)}} R \rTo^{\gamma} M \rTo 0
\end{equation}
where \(M\) is a right \(R\)-module.   Here the maps $\alpha$ and~$\beta$ are given by the matrices
\begin{equation*}
  \alpha =
  \begin{pmatrix}
    \id - \chi_{e_1} \\ %
    \id - \chi_{e_2}
  \end{pmatrix}
  \quad \text{and} \quad
  \beta =
  \begin{pmatrix}
    \id - \chi_{e_2} & -(\id - \chi_{e_1})
  \end{pmatrix}
  \ .
\end{equation*}

\begin{lemma}
  \label{lem:rows_exact}
  The complex~\eqref{eq:can-res-row} is exact.
\end{lemma}

\begin{proof}
  We allow \(M\) to be an arbitrary \(R_{(0,0)}\)-module initially.
  We have the direct sum decomposition
  \(M \tensor_{R_{(0,0)}} R = \bigoplus_{(i,j)} M \tensor_{R_{(0,0)}}
  R_{(i,j)}\); in fact, \(M \tensor_{R_{(0,0)}} R\) is a
  \(\bZ^2\)-graded \(R\)-module in this way. Any element
  $z \in M \tensor_{R_{(0,0)}} R$ can be expressed uniquely in the
  form
  \begin{equation*}
    z=\sum_{i,j \in \bZ} m_{i,j} \quad \text{where } m_{i,j} \in M \tensor_{R_{(0,0)}} R_{(i,j)} \ ,
  \end{equation*}
  such that $m_{i,j} = 0$ for almost all pairs~$(i,j)$. We say that
  {\it $z$~has $x$\nbd-amplitude in the interval~$[a,b]$\/} if
  $m_{i,j} =0$ for $i \notin [a,b]$.
  The {\it support of~$z$\/} is the (finite) set of
  all pairs~$(i,j)$ with $m_{i,j} \neq 0$.

  \smallskip

  After these initial comments, we proceed to verify exactness of the
  sequence.  We remark first that $\gamma$ is surjective since
  $\gamma(p \otimes 1)=p$. As a matter of fact,
  \(\sigma(p) = p \otimes 1\) defines an \(R_{(0,0)}\)-linear
  section~\(\sigma\) of~\(\gamma\).
  
  \smallskip

  Next, we show that $\alpha$ is injective. Let
  \(z=\sum_{i,j \in \bZ} m_{i,j}\) be an element of
  \(\ker(\alpha)\). Since
  \(\chi_{e_2}(m_{i,j}) \in M \tensor_{R_{(0,0)}} R_{(i,j+1)}\), the
  equality
  \begin{equation*}
    \alpha(z)=
    \begin{pmatrix}
      \displaystyle\sum_{i,j \in \bZ} \big(m_{i,j} - \chi_{e_1}(m_{i,j}) \big) \\[3.5ex]
      \displaystyle\sum_{i,j \in \bZ} \big(m_{i,j} - \chi_{e_2}(m_{i,j}) \big)
    \end{pmatrix}
    = 0
  \end{equation*}
  implies, by considering the homogeneous component of degree~\((i,j)\)
  of the second entry, that \(m_{i,j} - \chi_{e_2} (m_{i,j-1}) = 0\).
  If \(z \neq 0\) there exists \((i,j)\) in the support of~\(z\) such
  that \(m_{i,j-1} = 0\). For these indices we have
  \(m_{i,j} = \chi_{e_2} (m_{i,j-1}) = \chi_{e_2}(0) = 0\), a
  contradiction. This enforces \(z=0\) whence \(\alpha\) is injective.

  \smallskip

  To show that $\mathrm{im}\,\alpha = \ker\beta$ we will take an
  element~$(z_1,z_2)$ of $\ker\beta$ and show that it can be reduced
  to~$0$ by subtracting a sequence of elements of
  $\mathrm{im}\,\alpha \subseteq \ker\beta$. This clearly implies that
  $(z_1,z_2) \in \mathrm{im}\,\alpha$ as required.

  So let $(z_1,z_2) \in \ker\beta$, where
  \begin{equation*}
  z_1 = \sum_{i,j \in \bZ} m_{i,j} \quad \text{and} %
  \quad z_2 = \sum_{i,j \in \bZ} n_{i,j} \ , \quad %
  \text{with } m_{i,j}, n_{i,j} \in M \tensor_{R_{(0,0)}} R_{(i,j)} \ .
  \end{equation*}
  Choose integers $a$, $b$ and~$k$ such that $z_1$ has \(x\)-amplitude
  in~\([a,k]\) and \(z_2\) has \(x\)-amplitude in~\([a,b]\). If
  $k > a$ we define
  \begin{equation*}
    u = \sum_{j \in \bZ} \chi_{-e_1} (m_{k,j}) %
    \in M \tensor_{R_{(0,0)}} R \ , \ \ \text{with } %
    \chi_{-e_1} (m_{k,j}) \in M \tensor_{R_{(0,0)}} R_{(k-1,j)} \ ,
  \end{equation*}
  and set $(z_1',z_2') = (z_1,z_2) - \alpha(u)$. The
  \((k,j)\)-homogeneous component of \(z_1'\) vanishes by construction
  of~\(u\) and Corollary~\ref{cor:chi-commutes}, so that $z_1'$ has
  $x$\nbd-amplitude in $[a,k-1]$ while the \(x\)-amplitude of~\(z_2\)
  remains in~\([a,b]\). The element $(z_1',z_2')$ will be in
  $\mathrm{im}\,\alpha$ if and only if
  $(z_1,z_2) \in \mathrm{im}\,\alpha$.

  By iteration, we may thus assume that our initial pair $(z_1,z_2)$
  is such that the $x$\nbd-amplitude of~$z_1$ is in $\{a\} = [a,a]$.
  This actually necessitates \(z_2=0\). To see this, assume
  \(z_2 \neq 0\). We can then choose \((i,j)\) in the support
  of~\(z_2\) (so that in particular \(i \geq a\)) such that
  \((i+1,j)\) is not in the support of~\(z_2\), \ie, such that
  \(n_{i+1,j} = 0\). Since \(\beta (z_1,z_2) = 0\) we have,
  by considering the homogeneous component of degree~\((i+1,j)\),
  the equality
  \begin{equation*}
    m_{i+1,j} - \chi_{e_2} (m_{i+1, j-1}) - n_{i+1, j} + \chi_{e_1}
    (n_{i,j}) = 0 \ .
  \end{equation*}
  But the first three terms vanish, by choice of~\(i\) and our
  hypothesis on~\(z_1\), so that \(\chi_{e_1} (n_{i,j}) = 0\) whence
  \(n_{i,j} = 0\), contradicting the choice of~\(i\).
  
  Thus \(z_2=0\). This in turn implies that $\beta(z_1,0) = 0$ so that
  $m_{a,\ell} - \chi_{e_2} (m_{a,\ell-1}) = 0$ for any~\(\ell\). If
  $z_1$ is non-zero, we let $\ell$ be minimal with
  $m_{a, \ell} \neq 0$. Then
  $m_{a,\ell} = \chi_{e_2}(m_{a, \ell-1}) = \chi_{e_2}(0) = 0$, a
  contradiction. We conclude that
  $(z_1,z_2) = 0 \in \mathrm{im}\,\alpha$ as required.

  \smallskip

  To show $\mathrm{im}\,\beta = \ker \gamma$ it is enough to produce an
  \(R_{(0,0)}\)-linear map
  \begin{equation*}
    \phi \colon M \tensor_{R_{(0,0)}} R \rTo %
    \Big( M \tensor_{R_{(0,0)}} R \Big)^2
  \end{equation*}
  such that \(\beta \phi + \sigma \gamma = \id\), where \(\sigma\) is
  the section of~\(\gamma\) given by \(\sigma(p) = p \tensor 1\). It is
  enough to define~\(\phi\) on homogeneous primitive tensors
  \(x = m \tensor r\) with \(r \in R_{\rho}\), for the various
  \(\rho \in \bZ^2\).
  \begin{itemize}
  \item If \(a,b \geq 0\):
    \begin{equation*}
      \phi(x) = %
      \Big( - \sum_{k=1}^{b} \chi_{(0,-k)} (x) \qquad %
      \sum_{\ell=1}^{a} \chi_{(-\ell, -b)} (x) \Big) %
    \end{equation*}
  \item If \(a,b < 0\):
    \begin{equation*}
      \phi(x) = %
      \Big( \sum_{k=0}^{|b|-1} \chi_{(0,k)} (x) \qquad %
      - \sum_{\ell=0}^{|a|-1} \chi_{(\ell, |b|)} (x) \Big)%
    \end{equation*}
  \item If \(a< 0\) and \(b \geq 0\):
    \begin{equation*}
      \phi(x) = %
      \Big( - \sum_{k=1}^{b} \chi_{(0,-k)} (x) \qquad %
      - \sum_{\ell=0}^{|a|-1} \chi_{(\ell, |b|)} (x) \Big)%
    \end{equation*}
  \item If \(a \geq 0\) and \(b < 0\):
    \begin{equation*}
      \phi(x) = %
      \Big( \sum_{k=0}^{|b|-1} \chi_{(0,k)} (x) \qquad %
      \sum_{\ell=1}^{a} \chi_{(-\ell, -b)} (x) \Big) %
    \end{equation*}
  \end{itemize}
  In view of Corollary~\ref{cor:chi-commutes}, computing
  \(\beta \phi(m \tensor r)\), for \(r \in R\) a homogeneous element
  of degree~\(\rho\), results in a telescoping sum simplifying to
  \(m \tensor r - \chi_{-\rho} (m \tensor r)\). But
  \(\chi_{-\rho} (m \tensor r) = mr \tensor 1\) by
  Lemma~\ref{lem:chi_homogeneous} whence
  \(\chi_{-\rho} (m \tensor r) = \sigma \gamma (m \tensor r)\) so
  \(\beta \phi + \sigma \gamma = \id\) as required.
\end{proof}

\begin{proof}[Proof of Proposition~\ref{prop:canonical-resolution}]
  The map~\(\kappa\) is a quasi-isomorphism if and only if its mapping
  cone is acyclic. But this mapping cone is precisely the totalisation
  of the double complex~\eqref{eq:can-res-R-augmented}. As it is
  concentrated in a finite vertical strip, and as its rows are acyclic
  by Lemma~\ref{lem:rows_exact}, the totalisation is acyclic by
  standard results for double complexes.
\end{proof}

\section{The Mather trick}

Let $C$ be a chain complex of \(R\)-modules, and let \(D\) be a chain
complex of \(R_{(0,0)}\)-modules.  Suppose we have
\(R_{(0,0)}\)-linear chain maps $\alpha \colon C \rTo D$ and
$\beta \colon D \rTo C$ and a chain homotopy
$H \colon \id_C \simeq \beta \alpha$ such that
$dH+Hd=\beta\alpha-\id_C$. A calculation shows:

\begin{lemma}
  \label{lem:Mather-map}
  The matrix
  \begin{equation*}
      \begin{pmatrix}
        \alpha^* & 0 & 0 & 0 \\
        - \alpha^* \mu_1 H^*& \alpha^* & 0 & 0 \\
        - \alpha^* \mu_2 H^*& 0 & \alpha^* & 0 \\
        K&  \alpha^* \mu_2 H^*& - \alpha^* \mu_1 H^* & \alpha^*
      \end{pmatrix}
    \ ,
  \end{equation*}
  where we abbreviate
  \begin{equation*}
    K = \alpha^*(\mu_1 H^*\mu_2 - \mu_2 H^* \mu_1) H^*\colon C_n\tensor
    R\rTo D_{n+2} \tensor R \ ,
  \end{equation*}
  defines a chain map of algebraic tori
  \begin{equation*}
    \lambda \colon \torb{\id_C}{\id_C}{0} \rTo \torb{\alpha}{\beta}{H}
    \ .
    \tag*{\qedsymbol}
  \end{equation*}
\end{lemma}

\begin{theorem}[{\textsc{Mather} trick}]
  \label{thm:Mather-trick}
  If the map \(\alpha\) is a quasi-isomorphism then \(C\) is
  quasi-isomorphic to
  \(\torb{\alpha}{\beta}{H}\). More precisely, the maps
  \begin{equation*}
    C \lTo^{\kappa} \torb{\id_C}{\id_C}{0} \rTo^{\lambda}
    \torb{\alpha}{\beta}{H}
  \end{equation*}
  are quasi-isomorphism.
\end{theorem}

\begin{proof}
  The map~\(\kappa\) is a quasi-isomorphism by
  Lemma~\ref{prop:canonical-resolution}. The map~\(\lambda\), defined
  in Lemma~\ref{lem:Mather-map}, is a quasi-isomorphism if \(\alpha\)
  is since the representing matrix of~\(\lambda\) is lower triangular
  with diagonal terms~\(\alpha^*\). Note that
  \(\alpha^* = \alpha \tensor \id_R\) is a quasi-isomorphism since
  \(R\) is strongly graded and hence is a projective
  \(R_{(0,0)}\)-module.
\end{proof}

\section{Novikov homology}

Let \(K\) be a \(\bZ^2\)-graded \(R\)-module, with \(R\) a
\(\bZ^2\)-graded ring as usual. In analogy to~\eqref{eq:Novikov_type},
we can define
\begin{equation*}
  K_{*}[x,x\inv] \nov{y} = \bigcup_{n \geq 0} \prod_{y \geq -n}
  \bigoplus_{x \in \bZ} K_{(x,y)} \ ,
\end{equation*}
which is an \(\R[x,x\inv]\nov{y}\)-module in a natural way; similarly,
we can define
\begin{equation*}
  K_{*}\nov{x,y} = \bigcup_{n \geq 0} \prod_{x,y \geq -n}
  K_{(x,y)} \ ,
\end{equation*}
which is an \(\R\nov{x,y}\)-module in a natural way; and so on.

\medbreak

For \(R\)-modules of the form \(K = M \tensor_{R_{(0,0)}} R\), with
\(M\) an \(R_{(0,0)}\)-module, these constructions are close to the
usual induction functors; for example:

\begin{proposition}
  \label{noviso}
  Let $M$ be a right $R_{(0,0)}$-module. There is a natural map
  \begin{multline*}
    \Psi_M \colon M \tensor_{R_{(0,0)}} \R \nov{x,y} \rTo %
    (M \tensor_{R_{(0,0)}} R)_* \nov {x,y} \ , \\ %
    \quad m \tensor \sum_{x, y} r_{x, y} \mapsto
    \sum_{x, y} m\tensor r_{x, y}
  \end{multline*}
  which is an isomorphism if \(M\) is finitely presented.
\end{proposition}

\begin{proof}
  This can be verified by standard techniques: establish the result
  for \(M = R_{(0,0)}\) first (in which case it is trivial), extend to
  the case \(M = R_{(0,0)}^n\), and then treat the general case using
  a finite presentation of~\(M\). 
\end{proof}

From now on, assume that \(C\) is a bounded below complex of finitely
generated projective \(R\)-modules which is \(R_{(0,0)}\)-finitely
dominated. Thus we choose, once and for all, a bounded complex of
finitely generated projective \(R_{(0,0)}\)-modules \(D\) and mutually
inverse chain homotopy equivalences \(\alpha \colon C \rTo D\) and
\(\beta \colon D \rTo C\), and a chain homotopy \(H \colon \id_C
\simeq \beta \alpha\) so that \(dH+Hd = \beta \alpha - \id_C\).

The \textsc{Mather} trick~\ref{thm:Mather-trick} guarantees
that \(C\) and \(\torb{\alpha}{\beta}{H}\) are quasi-isomorphic; as
both \(R\)-module complexes are bounded below and consist of
projective \(R\)-modules, they are in fact chain homotopy equivalent.

\begin{lemma}
  \label{lem:nov-xy-acyclic}
  The chain complex \(C \tensor_{R} \R\nov{x,y}\) is contractible.
\end{lemma}

\begin{proof}
  By the previous remarks, it is enough to show that the chain complex
  \(\torb{\alpha}{\beta}{H} \tensor_{R} \R\nov{x,y}\) is
  acyclic. Recall that \(\torb{\alpha}{\beta}{H}\) is obtained from
  the diagram~\ref{diag:T} by a totalisation process, which commutes
  with induction. Thus we can realise \(\torb{\alpha}{\beta}{H}
  \tensor_{R} \R\nov{x,y}\) by totalising the following diagram:
  \begin{equation*}
    \begin{diagram}[labelstyle=\scriptstyle]
      D \tensor_{R_{(0,0)}} \R\nov{x,y} & \rTo^{\scriptstyle
        \id^*-\alpha^*\chi_{e_1}\beta^*}     & D \tensor_{R_{(0,0)}} \R\nov{x,y}                                                                                                           \\
      \dTo^{\id^*-\alpha^*\chi_{e_2}\beta^*} & \rdTo^{\alpha^*\chi_{e_1}H^*\chi_{e_2}\beta^*}_{-\alpha^*\chi_{e_2}H^*\chi_{e_1}\beta^*} & \dTo_{\id^*-\alpha^*\chi_{e_2}\beta^*} \\
      D \tensor_{R_{(0,0)}} \R\nov{x,y} & \rTo_{\id^*-\alpha^*\chi_{e_1}\beta^*}
      & D \tensor_{R_{(0,0)}} \R\nov{x,y}
    \end{diagram}
  \end{equation*}
  (We have used implicitly that the two functors
  \(\nix \tensor_{R_{(0,0)}} \R\nov{x,y}\) and
  \(\nix \tensor_{R_{(0,0)}} R \tensor_{R} \R\nov{x,y}\) are naturally
  isomorphic.) By Lemma~\ref{noviso}, we can re-write this
  further as the totalisation of the diagram
  \begin{equation}
    \label{diag:T-nov}
    \begin{diagram}[labelstyle=\scriptstyle]
      D_*\nov{x,y} & \rTo^{\scriptstyle
        \id^*-\alpha^*\chi_{e_1}\beta^*}     & D_*\nov{x,y}                                                                                                           \\
      \dTo^{\id^*-\alpha^*\chi_{e_2}\beta^*} &
      \rdTo^{\alpha^*\chi_{e_1}H^*\chi_{e_2}\beta^*}_{-\alpha^*\chi_{e_2}H^*\chi_{e_1}\beta^*}
      & \dTo_{\id^*-\alpha^*\chi_{e_2}\beta^*} \\
      D_*\nov{x,y} & \rTo_{\id^*-\alpha^*\chi_{e_1}\beta^*}
      & D_*\nov{x,y}
    \end{diagram}
  \end{equation}
  with maps suitably interpreted. To wit, an element~\(z\) of
  \(D_*\nov{x,y}\) (in some fixed chain level) has the form \(z =
  \sum_{x,y \geq a} z_{x,y}\) for some \(a \in \bZ\) and certain
  \(z_{x,y} \in D \tensor_{R_{(0,0)}} R_{x,y}\), and
  \begin{equation*}
    (\id^*-\alpha^*\chi_{e_1}\beta^*) (z) = \sum_{x,y \geq a} z_{x,y} %
    - \alpha^{*} \chi_{e_1} \beta^{*} (z_{x-1, y}) \ ,
  \end{equation*}
  see Lemma~\ref{lem:chi_homogeneous}; similar formul\ae{} can be
  written out for the other maps.

  The point is that \textit{the self-map
    \(\id^*-\alpha^*\chi_{e_1}\beta^*\) of \(D_*\nov{x,y}\) is an
    iso\-morphism with inverse given by the ``geometric series''}
  \begin{multline*}
    \qquad P(z) = \sum_{k \geq 0} \big( \alpha^* \chi_{e_1} \beta^* \big)^{k}
    (z) \\ = z + \alpha^* \chi_{e_1} \beta^* (z) + \alpha^* \chi_{e_1}
    \beta^* \alpha^* \chi_{e_1} \beta^* (z) + \ldots \ . \qquad
  \end{multline*}
  This follows immediately from the usual telescoping sum argument,
  once it is understood that the series actually defines a
  well-defined self-map of \(D_*\nov{x,y}\). But this is the case
  because \(\alpha^* \chi_{e_1} \beta^*\) maps an element of
  degree~\((x,y)\) to an element of degree~\((x+1,y)\).

  It is now a matter of computation to verify that the matrix
  \begin{equation}
    \label{eq:p}
    p =
      \begin{pmatrix}
        0 & P & 0 & 0 \\
        0 & 0 & 0 & 0 \\
        0 & 0 & 0 & -P \\
        0 & 0 & 0 & 0 \\
      \end{pmatrix}
  \end{equation}
  almost defines a contraction of the totalisation
  of~\eqref{diag:T-nov}, in the sense that \(dp + pd\) is an
  automorphism~\(q\) (in each chain level) --- in fact, \(q\) is given
  by a triagonal matrix with identity diagonal entries --- where \(d\)
  denotes the boundary map of the totalisation. It follows that
  \(q\inv p\) is a chain contraction as required.
\end{proof}

\begin{lemma}
  \label{lem:Laurent-Novikov_acyclic}
  The complex \(C \tensor_{R} \R[y,y\inv]\nov{x}\) is contractible.
\end{lemma}

\begin{proof}
  This is shown just like the previous Lemma.  It is enough to
  demonstrate that the chain complex
  \(\torb{\alpha}{\beta}{H} \tensor_{R} \R[y,y\inv]\nov{x}\) is
  acyclic; this complex can be written, using a variant of
  Lemma~\ref{noviso}, as the totalisation of the following diagram:
  \begin{equation}
    \begin{diagram}[labelstyle=\scriptstyle]
      D_*[y,y\inv]\nov{x} & \rTo^{\scriptstyle
        \id^*-\alpha^*\chi_{e_1}\beta^*}     & D_*[y,y\inv]\nov{x}                                                                                                           \\
      \dTo^{\id^*-\alpha^*\chi_{e_2}\beta^*} &
      \rdTo^{\alpha^*\chi_{e_1}H^*\chi_{e_2}\beta^*}_{-\alpha^*\chi_{e_2}H^*\chi_{e_1}\beta^*}
      & \dTo_{\id^*-\alpha^*\chi_{e_2}\beta^*} \\
      D_*[y,y\inv]\nov{x} & \rTo_{\id^*-\alpha^*\chi_{e_1}\beta^*}
      & D_*[y,y\inv]\nov{x}
    \end{diagram}
  \end{equation}
  \textit{The self-map \(\id^*-\alpha^*\chi_{e_1}\beta^*\)
    of \(D_*[y,y\inv]\nov{x}\) is an iso\-morphism with inverse}
  \(P(z) = \sum_{k \geq 0} \big( \alpha^* \chi_{e_1} \beta^*
    \big)^{k} (z)\);
  the matrix~\(p\) from~\eqref{eq:p}
  is such that \(dp + pd\) is an automorphism~\(q\) (in each chain
  level), where \(d\) denotes the boundary map of the totalisation.
  Consequently, \(q\inv p\) is a chain contraction.
\end{proof}

\begin{proof}[Proof of ``only if'' in Theorem~\ref{thm:main}]
  By Lemma~\ref{lem:nov-xy-acyclic}, the complex
  \(C \tensor_{R} \R\nov{x,y}\) is acyclic. Replacing \(R\) by the
  strongly \(\bZ^2\)-graded ring \(\bar R\) with
  \(\bar R_{(x,y)} = R_{(-x,y)}\), effectively substituting \(x\inv\)
  for~\(x\), gives that
  \(C \tensor_{\bar R} \bar\R\nov{x,y} = C \tensor_{R}
  \R\nov{x\inv,y}\) is acyclic. The remaining cases
  of~\eqref{eq:cond_vertex} are dealt with by similar re-indexing.

  By Lemma~\ref{lem:Laurent-Novikov_acyclic}, the complex
  \(C \tensor_{R} \R[y,y\inv]\nov{x}\) is acyclic. Replacing \(R\) by
  the strongly \(\bZ^2\)-graded ring \(\bar R\) with
  \(\bar R_{(x,y)} = R_{(-x,y)}\) gives that
  \(C \tensor_{\bar R} \bar\R[y,y\inv]\nov{x} = C \tensor_{R}
  \R[y,y\inv]\nov{x\inv}\) is acyclic. The other cases
  of~\eqref{eq:cond_edge} are dealt with by using
  \(\bar R_{(x,y)} = R_{(y,x)}\) and \(\bar R_{(x,y)} = R_{(-y,x)}\),
  respectively.
\end{proof}

\raggedright

\end{document}